\documentclass[12pt]{article}

\usepackage[latin1]{inputenc}
\usepackage{epsfig}
\usepackage{graphicx,psfrag}
\usepackage[tight]{subfigure}

\usepackage{amsfonts,amsthm,bbm,amssymb,epsf,amsmath,%
latexsym,amscd,amsfonts,enumerate,amsthm,supertabular,color,url,multirow}
\usepackage[colorlinks,bookmarksopen,bookmarksnumbered,citecolor=blue,urlcolor=blue]{hyperref}

\DeclareMathOperator{\rank}{rank}
\DeclareMathOperator{\diag}{diag}

\DeclareMathOperator{\supp}{supp}
\DeclareMathOperator{\spann}{span}
\DeclareMathOperator{\diam}{diam}
\DeclareMathOperator{\dist}{dist}

\DeclareMathOperator{\vol}{vol}

\newtheorem{theorem}{Theorem}
\newtheorem{lemma}{Lemma}
\newtheorem{example}{Example}
\newtheorem{remark}{Remark}
\newtheorem*{keywords}{Keywords}

\def\eps{\varepsilon}

\newcommand{\RR}{\mathbb{R}}
\newcommand{\ZZ}{\mathbb{Z}}
\newcommand{\NN}{\mathbb{N}}

\newcommand{\SSS}{\mathbb{S}}
\newcommand{\TT}{\mathbb{T}}

\def\cL{\mathcal{L}}

\def\cA{\mathcal{A}}
\def\cO{\mathcal{O}}
\def\cM{\mathcal{M}}
\def\cE{\mathcal{E}}

\def\cX{\mathcal{X}}

\def\hu{{\hat u}}

\def\matr#1{{\left[\begin{matrix} #1
\end{matrix}\right]}}

\def\Wl{W_{\ell}}

\begin{document}

\title{
Error Bounds for a Least Squares Meshless Finite Difference Method on Closed Manifolds}%

\author{Oleg Davydov\thanks{Department of Mathematics, University of Giessen, Arndtstrasse 2, 
35392 Giessen, Germany, \tt{oleg.davydov@math.uni-giessen.de}}
}%

\maketitle

\begin{abstract}
We present an error bound for a least squares version of the kernel based meshless finite difference method for elliptic
differential equations on smooth compact manifolds of arbitrary dimension without boundary. In particular, we obtain
sufficient conditions for the convergence of this method. Numerical examples are provided for the equation 
$-\Delta_\cM u + u = f$ on the 2- and 3-spheres, where $\Delta_\cM$ is the Laplace-Beltrami operator.
 \begin{keywords}
 RBF-FD, kernel based methods, PDE on manifolds, error bounds, meshless finite difference method, generalized finite differences
\end{keywords}
\end{abstract}

\section{Introduction}\label{intro}
Let $\cM$ be a 
 smooth compact manifold of dimension $d$ without boundary, 
 and let $L$ be an elliptic differential 
operator of order $2\kappa$, $\kappa\in \NN$, on $\cM$ with 
 infinitely differentiable coefficients 
 in local coordinates,  and a trivial null space. 
Then the equation 
\begin{equation}\label{pde}
L u=f
\end{equation}
has a unique solution $u$ in the Sobolev space $H^t(\cM)$, $t\in\RR$,  whenever $f\in H^{t-2\kappa}(\cM)$, and
there are constants $A,B>0$, depending only on $\cM,L,t$ and the choice of the Sobolev norms, such that 
\begin{equation}\label{reg}
A\|f\|_{H^{t-2\kappa}(\cM)}\le \|u\|_{H^t(\cM)}\le B\|f\|_{H^{t-2\kappa}(\cM)},
 \end{equation}
see e.g.\ \cite[Section 6]{Agranovich15}.

Meshless numerical methods are particularly attractive for solving operator equations on manifolds, in particular on the
sphere  \cite{FFprimer15,Freeden98}, because of the difficulties
of creating and maintaining suitable meshes or grids in this setting. 
Error bounds for meshless methods for the approximation of functions and specifically 
solutions of operator equations on manifolds have been studied for example in 
\cite{ChenLing20,CheungLing18,FuselierWright12,HNRW18,HNW10,LevesleyLuo03,LevesleyRagozin07,MortonNeamtu02,Narcowich95,NRW17,PhamTran14}. However, these
results do not apply to localized finite difference type methods considered below.

Meshless finite difference methods discretize a differential equation  \eqref{pde} 
(or similarly a boundary value problem)
on a set of irregular nodes $X=\{x_1,\ldots,x_n\}\subset\cM$  with the help of numerical differentiation formulas
\begin{equation}\label{ndf}
Lu(x_i)\approx\sum_{j\in J_i}w_{ij} u(x_j),\quad J_i\subset J:=\{1,\ldots,n\},\qquad i\in I,
\end{equation}
where the coefficients $w_{ij}\in\RR$ are obtained by requiring that the formula is exact for certain finite dimensional spaces
of functions, for example polynomials or kernel sums or a combination thereof, and the size of the sets of influence $X_i=\{x_j: j\in J_i\}$ is bounded by a 
fixed number $\nu\ll n$. In particular, the
kernel-based formulas are exact for all linear combinations
$\sum_{j\in J_i}c_j K(\cdot,x_j)$, $c_j\in\RR$, where $K:\cM\times\cM\to\RR$ is a positive definite kernel \cite{D21}. The
discrete approximate solution $\hu\in\RR^n$ of \eqref{pde}, such that $\hu_i\approx u(x_i)$, $i=1,\ldots,n$,
 is obtained by solving the sparse linear system
\begin{equation}\label{mfd}
\sum_{j\in J_i}w_{ij} \hu_j=f(x_i), \quad i\in I.
\end{equation}

Numerical performance of methods of this type 
has been studied 
e.g.\ in the book \cite{FFprimer15} and papers \cite{BFFB17,FLBWS12,LSH17,PLR18,Shankar17,SWKF15,SDKWT19online,SDT19,SuchdeKuhnert19}, in particular for
differential equations on the sphere and other manifolds.
In contrast to those meshless methods that discretize the weak form of the equations 
(see the surveys \cite{Belytschko96,NguRabBorDuf}), 
no integration of the trial functions over subdomains is needed, 
which is challenging in the meshless setting as the subdomains are not generated 
from a few reference shapes controlled by the mesh.
In contrast to the global collocation methods  (see e.g.~\cite{Kansa90b,PowerBarraco02}),
the linear systems are sparse. Moreover, by optimizing the selection of the sets of influence $J_i$, they may be made as  
sparse as they are in the mesh based methods, such as the finite element method, with comparable accuracy
\cite{D19arxiv,DavyOanh11,DOT20,DOT23,DavySaf21,ODP17}.  However, error bounds for meshless finite difference methods are
underdeveloped.

In the classical setting  $I=\{1,\ldots,n\}$, and hence \eqref{mfd} is a square linear system. 
An extension 
of the classical error analysis that requires that the system matrix of \eqref{mfd} is an $M$-matrix has been proposed 
in \cite{DKL84}, and conditions for the $M$-matrix property were investigated in \cite{DKL84,Seibold08}. However, this
approach only applies to special geometric configurations of the sets of influence with low convergence order, 
and no general theory exists. Moreover, except of the situations when the $M$-matrix property can be shown, no conditions
are known to guarantee the solvability of the square system \eqref{mfd}, even if the method performs very well in numerical
experiments.

By using more than $n$ numerical differentiation
formulas \eqref{ndf} we arrive at an overdetermined system, and $\hat u$ may be obtained by the least squares minimization.
In \cite{LSH17,TLH21} this situation arises from evaluating $Lu$ in the left hand side of \eqref{ndf} on a different set of nodes,
larger than $X$. This is similar to the ``overtesting'' mode that admits, for the collocation and other methods 
employing trial functions,  an error analysis under very general conditions
\cite{Schaback10,Schaback16}, which however does not cover 
finite difference type methods \cite{Schaback_error_analysis17}. 
The papers we just mentioned also discuss numerical and theoretical motivations for using overtesting 
instead of the classical setting of a square linear system.
An error analysis of the overtesting least squares method is provided in \cite{TLH21} with the help of  trial spaces defined
by piecewise radial basis functions over the Voronoi tessalation of the domain with respect to the nodes in $X$. The
overtesting is supposed to be dense enough, such that the discrete bilinear form that defines the least squares method
inherits coercitivity from the continuous bilinear form in a continuous least squares formulation of the boundary value
problem.


In this paper we derive error bounds for an overdetermined version of the meshless finite difference method
for elliptic differential equations on smooth closed manifolds, solved by least
squares, where the numerical differentiation formulas \eqref{ndf} are generated in blocks corresponding to a family of
overlapping subsets $J_\ell\subset J$, 
\begin{equation}\label{ndfmy}
[Lu(x_j)]_{j\in J_\ell}\approx W_\ell [u(x_j)]_{j\in J_\ell},\quad \ell=1,\ldots,m.
\end{equation}
The local differentiation matrices $W_\ell$ are invertible since they correspond to a reproducing kernel $K$ for a Sobolev
space $H^s(\cM)$, such that $L$ is positive and self-adjoint with respect to the inner product defined by $K$. This condition
in particular implies that the system matrix has full rank and hence there exists a unique least squares solution 
of the overdetermined linear system
\begin{equation}\label{mfdmy}
W_\ell [\hu_j]_{j\in J_\ell}=[f(x_j)]_{j\in J_\ell},\quad \ell=1,\ldots,m. 
\end{equation}

Under certain regularity assumptions that involve in particular the quasi-uniformity of $X$ and the existence of an 
atlas $\cA=\{(U_\ell,\varphi_\ell)\}_{\ell=1}^m$ with sufficiently nice $U_\ell\subset\cM$ 
such that $X\cap U_\ell=X_\ell:=\{x_j:j\in J_\ell\}$ and $X_\ell,X_k$ significantly
overlap whenever $U_\ell\cap U_k\ne\emptyset$, we show in Theorem~\ref{error_bound} that 
$$
\max_{j\in J}|u(x_j)-\hat u_j|=\mathcal{O}(h_\cA^{s-2\kappa-d-r_0}),
\qquad r_0:=\max\{0,\lfloor \tfrac{d}{2}\rfloor-2\kappa+1\},$$
and there exists $\tilde u\in H^{s+2\kappa}(\cM)$ such that $\tilde u(x_j)=\hu_j$, $j\in J$, and 
$$
\|u-\tilde u\|_{H^{2\kappa+r}(\cM)}=\mathcal{O}(h_\cA^{s-2\kappa-d-r}),\qquad 0\le r<s-\kappa,$$
where $h_\cA$ is the maximum diameter  of the sets $\varphi_\ell(U_\ell)$. This shows the pointwise 
convergence of the method for $h_\cA\to 0$ when $s>\max\{2\kappa,\tfrac{d}{2}+1\}+d$, and convergence in $H^{2\kappa}(\cM)$
when $s>2\kappa+d$.

The results apply in particular to the operators of the form $L=(-\Delta_\cM+\alpha I)^\kappa$, $\kappa\in\NN$,
$\alpha>0$, where $\Delta_\cM$ is the Laplace-Beltrami operator, $I$ is the identity, and $\cM=\SSS^d$ is the $d$-dimensional
sphere. In this case the differentiation matrices $W_\ell$ may be efficiently computed by employing  
restrictions to $\SSS^d$ of the Mat\'ern or Wendland
kernels for the ambient Euclidean space $\RR^{d+1}$. Numerical experiments in the case $\alpha=\kappa=1$ and $d\in\{2,3\}$
provided in Section~\ref{numerics} indicate higher convergence orders than those shown in Theorem~\ref{error_bound}.

The paper is organized as follows. Sections~\ref{prem} and \ref{aux} are devoted to auxiliary statements on various topics needed in the main parts
of the paper, such as equivalent norms for Sobolev spaces on manifolds, properties of  reproducing kernels and self-adjoint
differential operators, sampling inequalities and local differentiation matrices. 
Our least squares meshless finite difference method is presented in Section~\ref{method}, the error bounds in 
Section~\ref{estimates}, numerical examples in Section~\ref{numerics}, and a short conclusion in Section~\ref{conclusion}.


\section{Preliminaries}\label{prem}

We denote the partial derivative of a real function $u$ defined on a subset of $\RR^d$ by 
$\partial^\alpha u:=\frac{\partial^{|\alpha|}u}{\partial x_1^{\alpha_1}\cdots\partial x_d^{\alpha_d}}$, where 
$\alpha\in\ZZ_+^d$ and $|\alpha|:=\sum_{i=1}^d\alpha_i$. We will extensively use the Leibniz product rule and the estimates that follow from
the multivariate chain rule, see e.g.\ Sections 1.2, 1.63 and 3.41 in \cite{Adams03}.

The cardinality of a finite set $X$ will be denoted $\#X$, and $\partial S$ will stand for the boundary of 
a set $S$ in $\RR^d$ or $\cM$. We will use the usual notations $C(S)$, $C^\infty(S)$, $L^\infty(S)$, $L^2(S)$ for the spaces of continuous, infinitely differentiable,
essentially bounded or square integrable functions on appropriate subsets of $\RR^d$ or $\cM$.

Apart from the usual restriction  $f|_S$ for a subset $S$ of the domain of definition
of a function $f$, we denote by $f|_{X}$  the vector $[f(x)]_{x\in X}$ when $X$ is a finite set, 
and by $w|_I$  the vector $[w_j]_{j\in I}$, for any  $w=[w_j]_{j\in J}$ and a subset $I\subset J$. 

In what follows 
various ``constants'' denoted $C$, $C_1$, $C_2$, etc.\ will be different at different occurrences and depend on 
the manifold $\cM$ (and thus on $d$) on default.  We will explicitely list other parameters these constant may 
depend on in each case unless stated otherwise.

\subsection{Sobolev spaces on manifolds}\label{Sob}

Recall that in the case of integer $s\ge0$, Sobolev spaces $H^s(\Omega)$ are defined on any open set 
$\Omega$ as Hilbert spaces with the norm \cite{Adams03}
\begin{equation}\label{Sobint}
\|u\|_{H^s(\Omega)}=\Big(\sum_{k=0}^s|u|_{H^k(\Omega)}^2\Big)^{1/2},\quad 
|u|_{H^k(\Omega)}:=\Big(\sum_{|\alpha|= k}\|\partial^\alpha u\|_{L^2(\Omega)}^2\Big)^{1/2},
\end{equation}
in particular, $H^0(\Omega)=L^2(\Omega)$. When $s\in\RR\setminus\ZZ_+$ and $\Omega$ is a Lipschitz domain in 
$\RR^d$, we refer to \cite{Agranovich15} for the definition of $H^s(\Omega)$ as the restriction to $\Omega$ of the space 
$H^s(\RR^d)$ defined as the Bessel potential space. The Bessel norm for $H^s(\RR^d)$ is equivalent to 
$\|\cdot\|_{H^s(\RR^d)}$ of \eqref{Sobint}
when $s\in\ZZ_+$. However, the  definition via restriction, if applied to non-Lipschitz domains and
integer $s$, leads in general to smaller spaces due to the lack of extension of some functions with a finite norm 
\eqref{Sobint}, see e.g.~\cite[p.~287]{Burenkov98}. Even for Lipschitz domains, where extension theorems guarantee the equivalence of \eqref{Sobint}
to the restriction norm, the constants of equivalence depend on $\Omega$ and must be taken into account when they influence
estimates that have to be uniform with respect to $\Omega$. 

Following the standard definition of Sobolev spaces on smooth manifolds,  we pick a finite smooth atlas 
$\hat\cA=\{(\hat U_\ell,\hat\varphi_\ell)\}_{\ell=1}^{\hat n}$ and a smooth subordinate
partition of unity $\hat\Gamma=\{\hat\gamma_\ell\}_{\ell=1}^{\hat n}$, see e.g.~\cite{Lee13}, such that $\hat\gamma_\ell\in C^\infty(\cM)$, 
$\hat\gamma_\ell\ge0$, $\supp \hat\gamma_\ell\subset \hat U_\ell$,
$\sum_{\ell=1}^{\hat n}\hat\gamma_\ell=1$, and fix a norm for the Sobolev space $H^s(\cM)$, $s\in\RR$, in  the form
\begin{equation}\label{Snorm}
\|u\|_{H^s(\cM)}=\Big(\sum_{\ell=1}^{\hat n}\|u\hat\gamma_\ell\circ \hat\varphi_\ell^{-1}\|_{H^s(\RR^d)}^2\Big)^{1/2}.
\end{equation}

Below, when we speak of a norm for $H^s(\Omega)$ or $H^s(\cM)$ we will 
always assume that it is generated by an inner product. We reserve the notations $\|\cdot\|_{H^s(\Omega)}$, 
$\Omega\subset\RR^d$, and $\|\cdot\|_{H^s(\cM)}$, $\cM$ a manifold, for the specific norms defined in \eqref{Sobint} 
and \eqref{Snorm}, respectively. Different norms for  $H^s(\Omega)$ and $H^s(\cM)$ defined by reproducing kernels will 
be discussed in Section~\ref{pdk}.

\subsection{Reproducing kernels}\label{pdk} 

We first recall some basic facts about reproducing kernels and kernel based interpolation, see 
\cite{Aronszajn50,Buhmann03,Fasshauer07,FasshauerMcCourt15,PaulsenRaghupathi16,Wendland} for details, 
and then discuss reproducing kernels for Sobolev spaces.

Functions $K:\Omega\times \Omega\to\RR$ will be called \emph{kernels} on $\Omega$, where $\Omega$ is any set.
For a kernel $K$ and two sets $S,T$, $K|_{S,T}$ stands for the restriction of $K$ to 
$S\times T$, and for two finite sets $X,Y$ we also denote by 
$K|_{X,Y}$ the matrix $[K(x,y)]_{x\in X,\,y\in Y}$. 
To simplify the notation we set $K|_S:=K|_{S,S}$ and $K|_X:=K|_{X,X}$. 
For an operator $A$ that can be applied to one or both of the 
arguments of the kernel $K$, we write $A_1K$ or $A_2K$  in order to clarify whether $A$ is applied to 
the first or the second argument. We will also use the notation $K_A:=A_1K$.

A kernel $K:\Omega\times \Omega\to\RR$ is said to be \emph{symmetric} if $K(x,y)=K(y,x)$ for all $x,y\in\Omega$, and \emph{positive (semi-)definite} if 
the matrix $K|_X$ is positive (semi-)definite for  any finite subset $X$ of $\Omega$.


If $K$ is a symmetric positive semi-definite kernel on a set $\Omega$, then there is a unique
Hilbert space  $H_K=H_K(\Omega)$ of functions on $\Omega$ called \emph{native space} of $K$, with 
the inner product denoted by $(\cdot,\cdot)_K=(\cdot,\cdot)_{K,\Omega}$ and norm by 
$\|\cdot\|_{K}=\|\cdot\|_{K,\Omega}$, such that 
\begin{align}
\label{repr1}
K(\cdot,y)\in H_K&\quad\text{for all}\quad y\in\Omega,\\
\label{repr2}
(f,K(\cdot,y))_K=f(y)&\quad\text{for all}\quad f\in H_K,\quad y\in\Omega.
\end{align}
A Hilbert space $H$ of functions on $\Omega$ admits a kernel $K$ with these properties if and only if the linear functional 
$\delta_xf:=f(x)$ of point evaluation is bounded on $H$ for all $x\in\Omega$. Then $K$ is a \emph{reproducing kernel} of $H$, and $H$ is said to be a 
\emph{reproducing kernel Hilbert space}.

The identity \eqref{repr2} means that $K(\cdot,y)$ is the  Riesz representer of the linear functional $\delta_y\in H_K^*$.
Moreover, the kernel $K$ delivers Riesz representers for all bounded linear functionals.

\begin{lemma}[\cite{Wendland}, Theorem 16.7]\label{gammaK}
For any bounded linear functional $\gamma\in H_K^*$, the Riesz representer of $\gamma$ is the function
$\gamma_1K=\gamma_2K$ that belongs to $H_K$.
\end{lemma}


Assume that $K$ is positive definite. The \emph{kernel sums} of the form
\begin{equation}\label{kint}
\sigma=\sum_{j=1}^n c_j K(\cdot,x_j),\qquad c_j\in\RR,\quad x_j\in\Omega,
\end{equation}
are dense in $H_K(\Omega)$, and their kernel norm is given explicitely by 
\begin{equation}\label{snorm}
\|\sigma\|_{K}^2=\sum_{i,j=1}^nc_i c_j K(x_i,x_j).
\end{equation}

Given $X=\{x_1,\ldots,x_n\}\subset\Omega$, and data $a_1,\ldots,a_n\in\RR$, the kernel sum \eqref{kint}
is uniquely determined by the interpolation conditions $\sigma(x_i)=a_i$, $i=1,\ldots,n$. We will call it 
the \emph{kernel interpolant} of this data. In the case when $a_i=v(x_i)$, $i=1,\ldots,n$, for some
$v\in H_K$, we say that $\sigma$ is the kernel interpolant  of $v$. It satisfies 
\begin{equation}\label{orth}
(v-\sigma,\sigma)_{K}=0
\end{equation}
and hence
\begin{equation}\label{minnorm}
\|v-\sigma\|_{K}^2+\|\sigma\|_{K}^2=\|v\|_{K}^2.
\end{equation}

The restriction $K|_G$ on a subset $G\subset\Omega$ is obviously also a symmetric positive semi-definite kernel. 
Thanks to \eqref{snorm}, any kernel sum \eqref{kint} satisfies
\begin{equation}\label{srestr}
\|\sigma\|_{K,G}=\|\sigma\|_{K}\quad\text{whenever}\quad X\subset G.
\end{equation}
The native space of $K|_G$ can be described as follows.

\begin{lemma}[\cite{Aronszajn50}]\label{restrK}
If $K$ is a reproducing kernel on $\Omega$ and $G\subset\Omega$, then $K|_G$ is a reproducing kernel on $G$
with native space $H_K(G):=\{u|_G:u\in H_K(\Omega)\}$, and 
$$
\|u\|_{K,G}=\min\{\|\tilde u\|_{K,\Omega}:\tilde u\in H_K(\Omega),\,\tilde u|_G=u\}.$$
\end{lemma}

If $\Omega$ is an open set in $\RR^d$, then by Sobolev embedding $H^s(\Omega)$ is a reproducing kernel Hilbert space 
as long as $s>d/2$. Every (equivalent) norm for $H^s(\Omega)$ corresponds to a positive definite kernel on $\Omega$.

For $G\subset \Omega\subset\RR^d$,  the space $H^s(G)$ may be larger than the restriction of $H^s(\Omega)$
to $G$. Indeed, the existence of an extension of every $u\in H^s(G)$ to $\Omega$ with a finite $H^s$-norm is known to
be possible only under additional assumptions on the open set $G$. Even if this extension is available
for all $u\in H^s(G)$, we may need to know how much the $H^s$-norm may be increased as a result, and hence how much
the kernel norm $\|u\|_{K,G}$ may differ from the Sobolev norm $\|u\|_{H^s(G)}$. To quantify this,
for any pair of open sets $G\subset \Omega\subset\RR^d$, we denote by $\cE_s(G,\Omega)$  the 
\emph{extension constant}
$$
\cE_s(G,\Omega)
=\sup_{u\in H^s(G)\setminus\{0\}}\inf\{\|\tilde u\|_{H^s(\Omega)}/\|u\|_{H^s(G)}:
\tilde u\in H^s(\Omega),\, \tilde u|_{G}=u\}.$$
In the case $\Omega=\RR^d$ we write $\cE_s(G):=\cE_s(G,\RR^d)$.
The extension constant $\cE_s(G)$ is finite for Lipschitz domains due to well known extension theorems.


%
%
%
%

The next statement follows immediately from Lemma~\ref{restrK}.

\begin{lemma}\label{SobK} Let $\Omega,G$ be two open sets in $\RR^d$ with $G\subset\Omega$, and let $K$ be a 
reproducing kernel for $H^s(\Omega)$, where $s>d/2$. 
Then $H_K(G)\subset H^s(G)$ and
$$
\|u\|_{H^s(G)}\le C_1\|u\|_{K,G},\qquad u \in H_K(G).$$
Moreover, if the extension constant $\cE_s(G,\Omega)$ is finite, then 
$H_K(G)= H^s(G)$ and
$$
\|u\|_{K,G}\le C_2\cE_s(G,\Omega)\|u\|_{H^s(G)},\qquad u \in H_K(G).$$
The constants $C_1,C_2$ depend only on the constants of equivalence between 
$\|\cdot\|_{K,\Omega}$ and $\|\cdot\|_{H^s(\Omega)}$.
\end{lemma}

%

Any reproducing kernel for $H^s(\Omega)$ satisfies $K(x,\cdot)\in H^{s}(\Omega)$ for all $x\in \Omega$ 
because of \eqref{repr1}. However, with the help of Lemma~\ref{gammaK} we see that also derivatives 
$\partial^\beta_1 K(x,\cdot)$, with $|\beta|<s-d/2$, belong to $H^{s}(\Omega)$.

\begin{lemma}\label{smker}
Let $K:\Omega\times\Omega\to\RR$ be a reproducing kernel for $H^s(\Omega)$, $s>d/2$. Then
$\partial^\beta_1 K(x,\cdot)\in H^{s}(\Omega)$ for all $x\in\Omega$ and all $\beta\in\ZZ_+^d$ 
satisfying $|\beta|<s-d/2$.
\end{lemma}

\begin{proof} By Sobolev embedding (applied to a ball centered at $x$ and contained in $\Omega$) 
the linear functionals $\gamma$ of the form 
$\gamma v =\partial^\beta v(x)$, $v\in H^s(\Omega)$, are bounded on $H^s(\Omega)=H_K$ for all $x\in\Omega$ and 
$\beta\in\ZZ^d_+$ with $|\beta|<s-d/2$. Hence the statement follows by Lemma~\ref{gammaK}.
\end{proof}

If $\cM$ is a smooth closed manifold of dimension $d$, then $H^s(\cM)$ is a reproducing kernel Hilbert space when $s>d/2$.
It may however be difficult to compute the reproducing kernel such that its native space coincides with a given norm of 
$H^s(\cM)$. 
If an orthonormal basis is known for a Sobolev space, for example for $L^2(\cM)=H^0(\cM)$, then reproducing kernels can be constructed via
infinite series, see e.g.~\cite{LevesleyRagozin07,Narcowich95} and \cite[Section 17.4]{Wendland}. 
We discuss this in somewhat more detail in Section~\ref{difop}.

If $\cM$ is embedded into $\RR^m$, $m>d$, then 
in view of the trace theorems (see e.g.\ \cite[Theorem 2.3.7]{Agranovich15}) reproducing kernels
$K$ for $H^s(\cM)$, $s>d/2$, can be obtained by restricting to $\cM$ one of explicitely known positive definite kernels defined on $\RR^m$ with 
the native space norm equivalent to $H^{s+(m-d)/2}(\RR^m)$. 
The best known examples are Mat\'ern and Wendland kernels. 

The  Mat\'ern kernel is given by
\begin{equation}\label{Matern}
M_{s,d}(x,y)=\Phi_{s,d}(x-y),\quad
\Phi_{s,d}(x):=\tfrac{1}{2^{s-1}\Gamma(s)}\mathcal{K}_{s-d/2}(\|x\|_2)\|x\|_2^{s-d/2},\quad x\in\RR^m,
\end{equation}
where $\mathcal{K}_{\nu}$ denotes the modified Bessel function of second kind. 
The native space norm $\|\cdot\|_{\Phi_{s,d}}$ coincides with the Bessel potential 
norm of $H^{s+(m-d)/2}(\RR^m)$ since the $m$-dimensional Fourier transform 
of $\Phi_{s,d}$ is $(1+\|\omega\|_2^2)^{-s-(m-d)/2}$. By scaling $M_{s,d}$ with a \emph{shape parameter} $\eps>0$, we 
obtain the kernel $M_{s,d}^\eps(x,y)=M_{s,d}(\eps x,\eps y)$ whose native space also coincides with $H^{s+(m-d)/2}(\RR^m)$,
with a different but equivalent norm.

The native space of  the Wendland compactly supported kernel 
$W_{m,\ell}(x,y)=\phi_{m,\ell}(\|x-y\|_2)$,  $\ell\in\ZZ_+$, see \cite{Wendland}, 
 is norm equivalent to  $H^{s+(m-d)/2}(\RR^m)$
for $s=(d+1)/2+\ell$  (where $m\ge 3$ if $\ell=0$). Hence, in the case when $s-d/2\in \ZZ_++\frac12$, the restriction of
$W_{m,r}$, $m>d$, $r=\lfloor s-d/2 \rfloor=\lceil s-d/2\rceil-1$, is a reproducing kernel for $H^s(\cM)$.


\subsection{Differential operators}\label{difop}

Assuming that $\cM$ is a smooth compact manifold of dimension $d$ without boundary, we say that
a differential operator $L$ on $\cM$ has \emph{smooth coefficients} if its 
coefficients $a_{\ell,\alpha}$ in local coordinates,
$$
Lu\circ\hat\varphi_\ell^{-1} 
=\sum_{|\alpha|\le m}a_{\ell,\alpha}\partial^\alpha (u\circ\hat\varphi_\ell^{-1}),
\qquad u:\hat U_\ell\to\RR,\quad a_{\ell,\alpha}:\hat\varphi_\ell(\hat U_\ell)\to\RR,$$
are infinitely differentiable and uniformly bounded,
such that for some constants $C_{\alpha\beta}$,
\begin{equation}\label{aalpha}
\|\partial^\beta a_{\ell,\alpha}\|_{L^\infty(\hat\varphi_\ell(\hat U_\ell))}\le C_{\alpha\beta},\qquad 
\alpha,\beta\in\ZZ_+^d,\quad |\alpha|\le m,\quad \ell=1,\ldots,\hat n.
\end{equation}

Consider an  elliptic operator $L$ with smooth coefficients and  even order
$m=2\kappa$. In fact, all elliptic operators have even order if $d>2$ \cite{Agranovich15}. For all $t\in\RR$, 
 $L$ is a bounded linear operator from $H^t(\cM)$ to $H^{t-2\kappa}(\cM)$. 
Moreover, its inverse is a bounded linear operator from $H^{t-2\kappa}(\cM)$ to $H^t(\cM)$ as soon as the null space 
$N(L)=\{u: Lu=0\}$
of $L$ is trivial \cite[Theorem 6.2.1]{Agranovich15}, which implies \eqref{reg}. 

Since $H^{t_2}(\cM)$ is compactly embedded in
$H^{t_1}(\cM)$ for $t_1< t_2$ \cite[Theorem 2.3.1]{Agranovich15}, the operator $L^{-1}:H^t(\cM)\to H^t(\cM)$ is compact for
any $t\in R$. Hence the spectrum of $L$ consists of isolated eigenvalues of finite multiplicity with no finite
accumulation points. Assume that $L$, as an unbounded operator on $H^s(\cM)$ with domain
$H^{s+2\kappa}(\cM)$, is positive and self-adjoint with respect to an inner product $(\cdot,\cdot)_s$
in $H^s(\cM)$, that is, $(Lv,v)_s>0$ for all $v\in H^{s+2\kappa}(\cM)\setminus\{0\}$ and $(Lv,w)_s=(v,Lw)_s$ 
for all $v,w\in H^{s+2\kappa}(\cM)$. Then it follows from the spectral theorem 
(see e.g.~\cite[Theorem 7.17]{Weidmann80}, where real valued Hilbert spaces are considered)
that there exists an orthonormal basis of $H^s(\cM)$ consisting of eigenfunctions $e_j$, $j\in\NN$, of $L$, such that 
$$
Lv=\sum_{j=1}^\infty \lambda_jc_j e_j,\quad L^{-1}v=\sum_{j=1}^\infty \lambda^{-1}_jc_j e_j
\quad\text{for any}\quad v=\sum_{j=1}^\infty c_j e_j\in H^s(\cM),$$
where $\lambda_j>0$ are the eigenvalues of $L$ and $c_j\in\RR$. This implies for all $k\in\NN$,
$$
H^{s-km}(\cM)=\big\{\sum_{j=1}^\infty\lambda^k_jc_j e_j: \|c\|_2<\infty \big\},\quad
H^{s+km}(\cM)=\big\{\sum_{j=1}^\infty\lambda^{-k}_jc_j e_j: \|c\|_2<\infty \big\},$$
where $\|c\|_2^2:=\sum_{j=1}^\infty c^2_j$, and by interpolation (see e.g.~\cite[Section 13]{Agranovich15}),
$$
H^t(\cM)=\big\{\sum_{j=1}^\infty\lambda^{(s-t)/m}_jc_j e_j: \|c\|_2<\infty \big\},\qquad t\in\RR.$$
The norm in $H^t(\cM)$ defined by  $\|\sum_{j=1}^\infty\lambda^{(s-t)/m}_jc_j e_j\|_t:=\|c\|_2$ is equivalent to the standard
Sobolev norms, and $\lambda^{(s-t)/m}_j e_j$, $j=1,2,\ldots$, is an orthonormal basis with respect to this norm.
It is easy to see that $L$ is positive and self-adjoint with respect to the corresponding inner products $(\cdot,\cdot)_t$,
$t\in\RR$. Note that eigenfunctions $e_j$ belong to  
$H^t(\cM)$ for all $t\in\RR$, and thus to $C^\infty(\cM)$ 
since the eigenspaces of $L$ are independent of the choice of $s$ and inner product in $H^s(\cM)$.
The eigenvalues $\lambda_j$, numbered in nondecreasing order with multiplicities taken into account satisfy 
$\lambda_j=j^{2\kappa/d}(C+o(1))$, $j\to\infty$, where $C$ is a positive constant independent of $j$,
see \cite[Theorem 6.1.1]{Agranovich94}.

If $s>d/2$, then $H^s(\cM)$ is embedded into $C(\cM)$, and hence $H^s(\cM)$ is a reproducing kernel Hilbert space. 
It is easy to see that its kernel
corresponding to the inner product $(\cdot,\cdot)_s$ is given by
\begin{equation}\label{KfL}
K(x,y)=\sum_{j=1}^\infty e_j(x)e_j(y),
\end{equation}
where the series is a well defined real function on $\cM\times\cM$ since for each fixed $x\in\cM$ the sequence 
$e_j(x)=(v_x,e_j)$, $j=1,2,\ldots$, where $v_x\in H^s(\cM)$ is the Riesz representer of the point evaluation functional at $x$, is 
square summable. If $t>d/2$, then the kernel of $H^t(\cM)$ for the norm $\|\cdot\|_t$ is given by
\begin{equation}\label{Lt}
L^{2(s-t)/m}_1K(x,y)=\sum_{j=1}^\infty \lambda^{2(s-t)/m}_je_j(x)e_j(y),
\end{equation}
where the powers of $L$ are defined by
$$
L^\theta v=\sum_{j=1}^\infty \lambda^\theta_jc_j e_j,\qquad v=\sum_{j=1}^\infty c_j e_j,\qquad \theta\in\RR,$$
such that
$$
\|L^\theta v\|_t=\|v\|_{t+m\theta},\qquad t,\theta\in\RR,\quad v\in H^{t+m\theta}(\cM).$$
Since $\lambda_j>0$ for all $j$, the kernels $K_{L^\theta}=L^\theta_1K$ are positive definite whenever $s-\theta\kappa > d/2$.
%
%

We summarize in the following lemma the main findings of this section.

\begin{lemma}\label{KL}
Let $L$ be an elliptic differential operator  of order $2\kappa$ with smooth coefficients  and  trivial null space,
and let $K$ be a reproducing kernel for $H^s(\cM)$, where $s>d/2$. 
Assume that $L$ is positive and self-adjoint with respect to the inner product of $H^s(\cM)$ defined by $K$. 
Then $K_{L^\theta}=L^\theta_1K$ is a positive definite reproducing kernel for $H^{s-\theta\kappa}(\cM)$ for any $\theta\in\RR$ such that 
$s-\theta\kappa > d/2$. In particular, $K_{L^{-1}}$ is a reproducing kernel for $H^{s+\kappa}(\cM)$.
If $s>\kappa + d/2$, then  $K_L$ is a reproducing kernel for $H^{s-\kappa}(\cM)$ and
\begin{equation}\label{tK}
\|Lv\|_{K_L}=\|v\|_{K_{L^{-1}}},\qquad v\in H^{s+\kappa}(\cM).
\end{equation}
\end{lemma}


Note that Lemma~\ref{KL} also holds for certain classes of pseudodifferential operators on $\cM$, 
see \cite{Agranovich94,Agranovich15}.

\begin{lemma}\label{smkerL}
Under the hypotheses of Lemma~\ref{KL}, 
$K(\cdot,x)\in H^{s+2\kappa r}(\Omega)$ for all $x\in\Omega$, where $r$ is any integer satisfying $2\kappa r<s-d/2$. 
In particular, $K(\cdot,x)\in H^{s+2\kappa}(\Omega)$ for all $x\in\Omega$ if $s>2\kappa + d/2$.
\end{lemma}

\begin{proof} By Lemma~\ref{smker}, $L^r_1K(x,\cdot)\in H^{s}(\Omega)$ for all $x\in\Omega$ as soon as $2\kappa r<s-d/2$. 
Since the kernel  $L^r_1K$ is symmetric, it follows that $L^r_1K(\cdot,x)\in H^{s}(\Omega)$ for all $x\in\Omega$, and the statement 
is obtained by applying \eqref{reg} to $u=K(\cdot,x)$ and the operator $L^r$ of order $2\kappa r$.
\end{proof}

In the case of the $d$-dimensional sphere $\cM=\SSS^d$ of particular interest is a class of elliptic pseudodifferential operators 
\cite{Freeden98} that can be represented 
 in the form
\begin{equation}\label{pseu}
Lv=\sum_{j=1}^\infty \lambda_j c_je_j,\quad v=\sum_{j=1}^\infty  c_je_j,\qquad \lambda_j>0,\quad j\in\NN,
\end{equation}
where $e_j$ are the spherical harmonics, that is the eigenfunctions of the  Laplace-Beltrami operator $\Delta_{\cM}$,
normalized with respect to the standard inner product of $L^2(\SSS^d)$.  For  $d=2$ these operators have important applications in physical geodesy. 

In particular, operators of the form 
\begin{equation}\label{LB}
L=(-\Delta_{\cM}+\alpha I)^\kappa, \qquad \kappa\in\NN,\quad \alpha>0,
\end{equation}
where $I$ is the identity operator, 
are elliptic differential operators of order $2\kappa$ on $\SSS^d$ with smooth coefficients and trivial null spaces.
These operators satisfy \eqref{pseu} and are positive and self-adjoint with respect to the inner products
 for Sobolev spaces $H^s(\SSS^d)$, $s>d/2$, generated by reproducing kernels of the type
\begin{equation}\label{KfLS}
K(x,y)=\sum_{j=1}^\infty \mu_j e_j(x)e_j(y),
\end{equation}
with appropriately decaying positive real sequences $\{\mu_j\}_{j=1}^\infty$. Explicit \emph{zonal} kernels 
$K(x,y)=\psi(x^Ty)$ with this property are
obtained as restrictions to $\SSS^d$ of radially symmetric kernels (radial basis functions) 
$K(x,y)=\phi(\|x-y\|_2)$, $x,y\in\RR^{d+1}$ since 
$\|x-y\|_2=\sqrt{2-2x^Ty}$ if $\|x\|_2=\|y\|_2=1$, see e.g.~\cite[Section 17.2]{Wendland}. In particular, the hypotheses of 
Lemma~\ref{KL} are satisfied for the differential operators \eqref{LB} and kernels for $H^s(\SSS^d)$ obtained by 
restricting to $\SSS^d$ the Mat\'ern and Wendland kernels for $\RR^{d+1}$, as well as their scaled versions.

If $\cM$ is a Cartesian product of multiple spheres, possibly of different dimensions, then appropriate kernels, with respect to which 
the corresponding tensor products of the operators
\eqref{LB} are positive and self-adjoint, are obtained for $H^s(\cM)$ by taking tensor products of the 
spherical restrictions of the Mat\'ern and Wendland kernels. This follows from the fact that 
tensor products of kernels may be identified with the kernels of the tensor products of the respective 
reproducing kernel Hilbert spaces, see e.g.~\cite{PaulsenRaghupathi16}. This applies in particular to the $d$-dimensional torus
$\cM=\TT^d$ \cite{Narcowich95}. 

Differential operators \eqref{LB} on arbitrary Riemannian manifolds also satisfy the assumptions of Lemma~\ref{KL}, see \cite{Agranovich15}.
However, closed form formulas for the kernels of the form \eqref{KfLS}, where $e_j$ are the eigenfunctions of $\Delta_\cM$, are not available in general.
By considering the flat torus $\cM=\RR^d/\ZZ^d$ we may find appropriate kernels for the periodic boundary value problems for the 
operators $L=(-\Delta+\alpha I)^\kappa$ on the unit cube, where $\Delta$ is the $d$-dimensional Laplace operator. In the univariate case $d=1$ this kernel is known in terms
of Bernoulli polynomials, see e.g.~\cite[Section 3.2]{FasshauerMcCourt15}.

%
%
%


\section{A meshless finite difference method}\label{method}

Assume that $L$ is an elliptic differential operator on $\cM$ of order $2\kappa$ with smooth coefficients  
and trivial null space, and  $K$ is a reproducing kernel for $H^s(\cM)$, where $s>\kappa+d/2$, such that 
$L$ is positive and self-adjoint with respect to the inner product $(\cdot,\cdot)_K$ of $H^s(\cM)$ defined by $K$. 

We consider the equation 
$Lu=f$, with continuous $f$, and look for a discrete solution $\hu\in\RR^n$  to approximate $u|_X$ 
on a given set of \emph{nodes} $X=\{x_1,\ldots,x_n\}\subset\cM$.

We choose a system  $\cX=\{X_\ell:\ell =1,\ldots,m\}$ of subsets 
$X_\ell =\{x_j:j\in J_\ell \}\ne\emptyset$, with $J_\ell \subset J:=\{1,\ldots,n\}$, such
that   $\bigcup_{\ell =1}^m X_\ell =X$. Let $n_\ell:=\# X_\ell$.
Consider the local differentiation matrices 
$\Wl:=W_{X_\ell}\in\RR^{n_\ell\times n_\ell}$ of \eqref{difm}, that is
\begin{equation}\label{Wl}
\Wl=K_L|_{X_\ell}(K|_{X_\ell})^{-1}
\qquad \ell=1,\ldots,m.
\end{equation}
The kernel $K_L$ is well defined and positive definite by Lemma~\ref{KL}, and hence the matrices
 $\Wl$ are non-singular.

We now set up an overdetermined linear system 
\begin{equation}\label{ods}
W\!Mv=Mf|_X
\end{equation}
with unknown vector $v\in\RR^N$, the block diagonal matrix
$$
W=\diag(W_1,\ldots,W_m):=
\matr{W_1 & \cdots & 0\\ 
\vdots & \ddots &\vdots\\
0 &\cdots & W_m}\in \RR^{N\times N},\quad N:=n_1+\cdots+n_m,$$
and the incidence matrix 
$$
M=\matr{M_1\\ \vdots \\ M_m}\in\RR^{N\times n},\quad 
M_\ell=[\delta_{ij}]_{i\in J_\ell,j\in J}\in\RR^{n_\ell\times n},\quad
\delta_{ij}=\begin{cases} 1 & \text{if }i=j,\\
0 & \text{otherwise.}
\end{cases}$$ 
The matrix $W\!M$ has the full rank $n$ since $W$ is non-singular and $\rank M=n$. 

We determine the discrete solution $\hat u\in\RR^n$ as the unique least squares solution of \eqref{ods}, such that
\begin{equation}\label{leastsq}
\|W\!M\hat u-Mf|_X\|_2=\min_{v\in\RR^n}\|W\!Mv-Mf|_X\|_2.
\end{equation}

Note that \eqref{ods} can be more explicitely written as
\begin{equation}\label{ods1}
W_\ell \,v|_{J_\ell} = f|_{X_\ell},\quad \ell=1,\ldots,m,
\end{equation}
and hence, in line with the main idea of meshless generalized finite differences,
 each linear equation in \eqref{ods} is derived from a numerical differentiation formula
$$
Lu(x_i)\approx \sum_{j\in J_\ell} w^{(\ell)}_{i,j}u(x_j),\quad i\in J_\ell,$$
where the weight vector $[w^{(\ell)}_{i,j}]_{j\in J_\ell}$ is the row of $W_\ell$ corresponding to the node $x_i$.
Since 
\begin{equation}\label{ods2}
\|W\!Mv-Mf|_X\|_2^2=\sum_{\ell=1}^m \|W_\ell \,v|_{J_\ell} - f|_{X_\ell}\|_2^2,
\end{equation}
we in fact minimize in \eqref{leastsq} the sum of squared residuals of the numerical differentiation on the subsets
$X_\ell$.

Even if this method is well defined for any arbitrarily chosen nodes $X\subset\cM$ and any system $\cX$ of their subsets
satisfying $\bigcup_{\ell =1}^m X_\ell =X$, we have in mind  $X_\ell$ being clusters of neighboring nodes, with a significant overlap
between neighboring clusters.
In Section~\ref{estimates} we present an error bound for $\|u|_X-\hu\|_\infty$ under certain assumptions on $X$ and $\cX$, that in
particular quantify the overlaps in terms of certain determining properties of the images of the non-empty intersections $X_\ell\cap X_p$
in the charts of an atlas for $\cM$. 

For a practical realization of the method we suggest to choose a set $Y=\{y_1,\ldots,y_m\}\subset \cM$, possibly
$Y\subset X$, and select each $X_\ell$ as the set of $n_\ell$ nearest nodes to $y_\ell$, either in the sense of a properly defined distance
on $\cM$, or by the distance in an ambient Euclidian space. The latter approach with $Y=X$ is used in numerical examples of 
Section~\ref{numerics} for $\cM=\SSS^d$, $d\in\{2,3\}$, with $\RR^{d+1}$ as ambient space.

\begin{remark}\label{pnorm}\rm 
Instead of \eqref{leastsq} we may in principle define 
$\hat u\in\RR^n$ by minimizing a different $p$-norm, 
\begin{equation}\label{pmin}
\|W\!M\hat u-Mf|_X\|_p=\min_{v\in\RR^n}\|W\!Mv-Mf|_X\|_p,\quad 1\le p\le\infty,
\end{equation}
see also Remark~\ref{pnorm_est} below.
The 2-norm is preferable from the point if view of computational efficiency. In particular, for large $n$ and $m$ and small
$n_\ell$ we may use existing efficient methods for sparse linear least squares, see  \cite{Bjorck96}.
Both cases $p=1$ and $\infty$ have the interesting property that the solution $\hu$ of \eqref{pmin} satisfies certain linear
systems with a matrix $A$ consisting of relatively few rows of $WM$, and vector $b$ of corresponding components of 
$Mf|_X$. Namely, $A$ is an $n\times n$-matrix with $A\hu=b$ when $p=1$ \cite{Davies67}, 
and an $(n+1)\times n$-matrix with all components of the residual $A\hu-b$ equal in absolute value
 when $p=\infty$ \cite[p.~36]{Cheney82}. In both cases this means that  $\hu$ may be obtained by
minimizing $\|Av-b\|_p$ as soon as $A$ has full rank. This size reduction of the system \eqref{ods} is however 
computationally demanding.

\end{remark}


\section{Auxiliary results}\label{aux}
Before formulating and proving in Section~\ref{estimates} the error bounds for the above method we provide 
several technical results needed in the proof.

\subsection{Admissible atlases}\label{adm}

Recall that we have chosen and fixed a finite smooth atlas 
$\hat\cA=\{(\hat U_\ell,\hat\varphi_\ell)\}_{\ell=1}^{\hat n}$ of $\cM$ and a smooth subordinate
partition of unity $\hat\Gamma=\{\hat\gamma_\ell\}_{\ell=1}^{\hat n}$, 
and consider the norms for the Sobolev spaces $H^s(\cM)$ in the form \eqref{Snorm} 
dependent on this choice.

It follows from the properties of $\hat\gamma_\ell$ that the sets $\hat U^{c}_\ell:=\{x\in\cM:\hat\gamma_\ell(x)>c\}\subset \hat U_\ell$, $\ell=1,\ldots,\hat n$, form an open cover 
of $\cM$ when $c=0$, and hence also for all sufficiently small $c>0$. 
We say that a finite smooth atlas $\cA=\{(U_\ell,\varphi_\ell)\}_{\ell=1}^{m}$ for $\cM$ is \emph{$c$-admissible}, where $c>0$, if for each $\ell=1,\ldots,m$ there exists
$k=k(\ell)$ such that $U_\ell\subset \hat U^{c}_k$ and $\varphi_\ell=\hat\varphi_k|_{U_\ell}$. An atlas $\cA$ is
\emph{admissible} if it is \emph{$c$-admissible} for some $c>0$.
The \emph{covering number}
of $\cA$ is defined as
\begin{equation}\label{muA}
\mu_\cA:=\max_{1\le \ell\le m}\#\{j: U_j\cap U_\ell\ne\emptyset\}.
\end{equation}

We show that the Sobolev norm of a function $u$ on $\cM$ may be characterized in terms of Sobolev norms of the compositions
$u\circ \varphi_\ell^{-1}$ for any $c$-admissible atlas, without multiplying $u$ by the partition of unity function 
$\hat\gamma_\ell$ as in \eqref{Snorm}.

\begin{lemma}\label{admis} Let $\cA=\{(U_\ell,\varphi_\ell)\}_{\ell=1}^{m}$ be a $c$-admissible atlas for $\cM$ for some $c>0$. Then for all integer $s\ge0$,
\begin{equation}\label{aSob} 
C_1\|u\|_{H^s(\cM)}\le \Big(\sum_{\ell=1}^{m}\|u\circ \varphi_\ell^{-1}\|_{H^s(\varphi_\ell(U_\ell))}^2\Big)^{1/2}
\le C_2\sqrt{\mu_\cA}\,\|u\|_{H^s(\cM)},\quad u\in H^s(\cM),
\end{equation}
where $C_1>0$ depends only on $s$,  the choice of the atlas $\hat\cA$ 
and the partition of unity $\hat\Gamma$, and 
$C_2$ depends in addition on $c$. 
\end{lemma}

\begin{proof} 
Let $L_k:=\{\ell: U_\ell\cap \hat U_k\ne\emptyset\}$. Then $\hat U_k\subset \bigcup_{\ell\in L_k}U_\ell$
and hence
$$
\|u\hat\gamma_k\circ \hat\varphi_k^{-1}\|_{H^s(\RR^d)}^2
=\|u\hat\gamma_k\circ \hat\varphi_k^{-1}\|_{H^s(\hat\varphi_k(\hat U_k))}^2
\le\sum_{\ell\in L_k}\|u\hat\gamma_k\circ \hat\varphi_k^{-1}\|_{H^s(\hat\varphi_k(U_\ell\cap \hat U_k))}^2.
$$
Since $\hat\varphi_k^{-1}=\varphi_\ell^{-1}\circ\hat\varphi_{k(\ell)}\circ\hat\varphi_k^{-1}$ on
$\hat\varphi_k(U_\ell\cap \hat U_k)$, we have by the Leibniz and the chain rules,
$$
\|u\hat\gamma_k\circ \hat\varphi_k^{-1}\|_{H^s(\hat\varphi_k(U_\ell\cap \hat U_k))}^2
\le C\|u\circ \varphi_\ell^{-1}\|_{H^s(\varphi_\ell(U_\ell\cap \hat U_k))}^2,$$
where $C$ depends only on the size of derivatives of order up to $s$ of the functions $\hat\gamma_k\circ \hat\varphi_k^{-1}$ and the transition maps 
$\hat\varphi_{j}\circ\hat\varphi_{k}^{-1}$ of $\hat\cA$, and the (nonzero) determinants of their Jacobi matrices.
Finally,
$$
\sum_{\ell=1}^m\sum_{\substack{k=1\\ \ell\in L_k}}^{\hat n}\|u\circ \varphi_\ell^{-1}\|_{H^s(\varphi_\ell(U_\ell\cap \hat U_k))}^2
\le \hat n\sum_{\ell=1}^m\|u\circ \varphi_\ell^{-1}\|_{H^s(\varphi_\ell(U_\ell))}^2,$$
and we deduce the first inequality in \eqref{aSob}.

On the other hand, for each $k=1,\ldots,\hat n$,
$$
\sum_{\substack{\ell=1\\ k(\ell)=k}}^{m}\|u\circ \varphi_\ell^{-1}\|_{H^s(\varphi_\ell(U_\ell))}^2
\le \mu_\cA \|u\circ \hat\varphi_k^{-1}\|_{H^s(\hat\varphi_k(\hat U^{c}_k))}^2.$$
By applying the Leibniz product rule to $\frac{1}{\hat\gamma_k}\cdot(u \hat\gamma_k)$ and estimating the $L^2$-norm of the products of (bounded) derivatives of 
$\frac{1}{\hat\gamma_k\circ \hat\varphi_k^{-1}}$ with derivatives of $u\circ\hat\varphi_k^{-1}$, we obtain
\begin{equation}\label{loc1}
\|u\circ \hat\varphi_k^{-1}\|_{H^s(\hat\varphi_k(\hat U^{c}_k))}\le C \|u\hat\gamma_k\circ \hat\varphi_k^{-1}\|_{H^s(\RR^d)},
\end{equation}
where $C$ depends only on $\hat\cA$, $\hat\Gamma$, $c$ and $s$.
This completes the proof of the second inequality in \eqref{aSob}.
\end{proof}

Note that \eqref{aSob} holds without any assumptions about the boundaries of the sets $\varphi_\ell(U_\ell)$. In particular, they do not need to be Lipschitz domains.
If they are, then we can extend the functions $u|_{U_\ell}$ to $\cM$, which will be useful 
for establishing double estimates for $\|u\|_{H^s(\cM)}$ in terms of the native space norms of a reproducing kernel 
in charts, see \eqref{admisKU} below.

\begin{lemma}\label{ext} Let $U\subset \cM$ be a domain such that $U\subset \hat U^c_k$ for  
some $c>0$ and $k$, and the extension constant $\cE_s(\hat\varphi_k(U))$ is finite
for an integer $s\ge0$. 
Then for any $u:U\to\RR$ such that $u\circ \hat\varphi_k^{-1}\in H^s(\hat\varphi_k(U))$ 
there exists a function $\tilde u\in H^s(\cM)$ such that $\tilde u|_{U}=u$ and
$$
\|\tilde u\|_{H^s(\cM)}\le C\cE_s(\hat\varphi_k(U))\|u\circ \hat\varphi_k^{-1}\|_{H^s(\hat\varphi_k(U))},$$
where $C$ depends only on $\hat\cA,\hat\Gamma,s,c$.
\end{lemma}

\begin{proof} Let $w\in H^s(\RR^d)$ be an extension of $u\circ \hat\varphi_k^{-1}$ such that 
$$
\|w\|_{H^s(\RR^d)}\le 2 \cE_s(\hat\varphi_k(U)) \|u\circ \hat\varphi_k^{-1}\|_{H^s(\hat\varphi_k(U))}.$$
We define $\tilde u$ as follows,
$$
\tilde u(x)=
\begin{cases}
(\chi_{k,c} w\circ \hat\varphi_{k})(x), & \text{if } x\in \hat U_{k},\\
0, &\text{otherwise,}
\end{cases} $$
where $\chi_{k,c}\in C^\infty(\RR^d)$ is a cutoff function, see e.g.~\cite{Lee13}, that satisfies
$$
\chi_{k,c}(x)=1,\quad x\in \hat\varphi_k(\hat U^c_k),\quad \text{and}\quad \chi_{k,c}(x)=0,\quad x\notin \hat\varphi_k(\hat U_k).$$
Then $\tilde u|_{U}=u$, and for any $\ell$ by the Leibniz and the chain rules,
\begin{align*}
\|\tilde u\hat\gamma_\ell\circ \hat\varphi_\ell^{-1}\|_{H^s(\RR^d)}
&=\|(\hat\gamma_\ell\circ \hat\varphi_\ell^{-1})
(\chi w\circ \hat\varphi_k\circ \hat\varphi_\ell^{-1})\|_{H^s(\hat\varphi_\ell(\hat U_k\cap\hat U_\ell))}\\
&\le C_1\|w\|_{H^s(\hat\varphi_k(\hat U_k\cap\hat U_\ell))},
\end{align*}
where $C_1$ depends only on $c,s$ and the choice of  $\hat\cA$, $\hat\Gamma$ and the cutoff function $\chi_{k,c}$. 
Clearly, cutoff functions $\chi_{k,c}$, $k=1,\ldots,\hat n$, may be chosen once and for all as soon as other parameters are fixed. Hence
$$
\|\tilde u\|_{H^s(\cM)}^2=\sum_{\ell=1}^{\hat n}\|\tilde u\hat\gamma_\ell\circ \hat\varphi_\ell^{-1}\|_{H^s(\RR^d)}^2
\le 4\hat n C_1^2\cE^2_s(\hat\varphi_k(U))\|u\circ \hat\varphi_k^{-1}\|_{H^s(\hat\varphi_k(U))}^2.\qedhere$$
\end{proof} 


If $K$ is a reproducing kernel on $\cM$, 
$U$ an open set in $\cM$, and $\varphi:U\to\RR^d$ a homeomorphism from $U$ to 
$\varphi(U)\subset\RR^d$,
then 
$K^\varphi(x,y):=K(\varphi^{-1}(x),\varphi^{-1}(y))$, $x,y\in \varphi(U)$,
is obviously the reproducing kernel for the Hilbert space
\begin{equation}\label{Hphi}
H_K(\varphi(U))=\{u\circ\varphi^{-1}:u \in H_K(U)\}
\end{equation}
isometric to $H_K(U)$, with the norm given by
$$
\|u\circ\varphi^{-1}\|_{K^\varphi,\varphi(U)}=\|u\|_{K,U},$$
where we use the notation introduced in Lemma~\ref{restrK}.
We show an analogue of Lemma~\ref{SobK}, relating the native space norm on $U$ to the Sobolev norm on its 
chart image in $\RR^d$ in the case when $K$ is a reproducing kernel for a Sobolev space.

\begin{lemma}\label{SobKm} Let $K$ be a reproducing kernel for $H^s(\cM)$, with an integer $s>d/2$, and let
 $U\subset \cM$ be a domain such that $U\subset \hat U^c_k$ for  
some $c>0$ and $k$. Then $H_K(\hat\varphi_k(U))\subset H^s(\hat\varphi_k(U))$ and
$$
\|u\circ \hat\varphi_k^{-1}\|_{H^s(\hat\varphi_k(U))}\le C_1\|u\|_{K,U},\qquad u\in H_K(U).$$
Moreover, if the extension constant $\cE_s(\hat\varphi_k(U))$ is finite, then 
$H_K(\hat\varphi_k(U))= H^s(\hat\varphi_k(U))$ and
$$
\|u\|_{K,U}\le C_2\cE_s(\hat\varphi_k(U))\|u\circ \hat\varphi_k^{-1}\|_{H^s(\hat\varphi_k(U))},
\qquad u\in H_K(U).$$
The constants $C_1,C_2$ depend only on $K$, $\hat\cA$, $\hat\Gamma$, $c$ and $s$.
\end{lemma}

\begin{proof} Let $v\in H_K(\hat\varphi_k(U))$, that is $v=u\circ \hat\varphi_k^{-1}$ for some $u\in H_K(U)$. 
Then $u$ is the restriction to $U$ of a function in 
$H_K(\cM)$, denoted $u$ again. By Lemma~\ref{restrK} we assume without loss of generality that
$\|u\|_{K,U}=\|u\|_{K,\cM}$.
As in the proof of Lemma~\ref{admis}, we obtain by \eqref{loc1},
$$
\|u\circ \hat\varphi_k^{-1}\|_{H^s(\hat\varphi_k(U))}\le\|u\circ \hat\varphi_k^{-1}\|_{H^s(\hat\varphi_k(\hat U^{c}_k))}
\le C \|u\hat\gamma_k\circ \hat\varphi_k^{-1}\|_{H^s(\RR^d)}\le C \|u\|_{H^s(\cM)},$$
where $C$ depends only on $\hat\cA$, $\hat\Gamma$, $c$ and $s$. Since the norms 
$\|\cdot\|_{K,\cM}$ and $\|\cdot\|_{H^s(\cM)}$ are equivalent, the first assertion of the lemma follows.

We now assume that $\cE_s(\hat\varphi_k(U))<\infty$. For any $v\in H^s(\hat\varphi_k(U))$, the function 
$u=v\circ \hat\varphi_k$ is by Lemma~\ref{ext} the restriction to $U$ of a function $\tilde u\in H^s(\cM)$
such that
$$
\|\tilde u\|_{H^s(\cM)}\le C\cE_s(\hat\varphi_k(U))\|v\|_{H^s(\hat\varphi_k(U))},$$
where $C$ depends only on $c,s,\hat\cA,\hat\Gamma$. Since $H^s(\cM)=H_K(\cM)$ with equivalent norms, 
it follows from Lemma~\ref{restrK} that $\|u\|_{K,U}\le C\|\tilde u\|_{H^s(\cM)}$, where $C$ depends only on 
$K,\hat\cA,\hat\Gamma$. Hence $v\in H_K(\hat\varphi_k(U))$, and 
the second assertion follows.
\end{proof}

Given a $c$-admissible atlas $\cA=\{(U_\ell,\varphi_\ell)\}_{\ell=1}^{m}$  for $\cM$, with $c>0$, we
obtain from Lemma~\ref{SobKm}, for any $u\in H^s(\cM)$ and $\ell=1,\ldots,m$,
$$
C_1^{-1}\|u\circ \varphi_\ell^{-1}\|_{H^s(\varphi_\ell(U_\ell))}\le \|u|_{U_\ell}\|_{K,U_\ell}
\le C_2\cE_s(\varphi_\ell(U_\ell))\|u\circ \varphi_\ell^{-1}\|_{H^s(\varphi_\ell(U_\ell))},$$
with constants $C_1,C_2$ of Lemma~\ref{SobKm}.
Hence, it follows by Lemma~\ref{admis} that
\begin{equation}\label{admisKU}
C_1\|u\|_{H^s(\cM)}\le \Big(\sum_{\ell=1}^{m}\|u|_{U_\ell}\|_{K,U_\ell}^2\Big)^{1/2}
\le C_2\sqrt{\mu_\cA}\max_{\ell=1,\ldots,m}\cE_s(\varphi_\ell(U_\ell))\,\|u\|_{H^s(\cM)},
\end{equation}
for all $u\in H^s(\cM)$, with constants $C_1,C_2$ depending only on $K$, $\hat\cA$, $\hat\Gamma$, $c$ and $s$.

Since $\|u\|_{H^s(\cM)}$ and the native space norm $\|u\|_{K}$ are equivalent, we also get the same double estimate
with $\|u\|_{H^s(\cM)}$ replaced by $\|u\|_{K}$. Thus, a detour via Sobolev spaces delivers a relation 
between $\|u\|_{K}$ and $\|u|_{U_\ell}\|_{K,U_\ell}$, $\ell=1,\ldots,m$.

\subsection{Extension constants}\label{extc}

Since we will need to bound the extension constants $\cE_s(\hat\varphi_k(U))$ uniformly for families of local subdomains
$U\subset\cM$, compare \eqref{admisKU}, we consider the question of estimating $\cE_s(\Omega)$ for Lipschitz domains $\Omega\subset\RR^d$. 

A bounded domain $\Omega$ satisfies the Lipschitz condition if it has a locally Lipschitz boundary, that is, each point on the
boundary of $\Omega$ has a neighborhood whose intersection with $\partial\Omega$ is the graph of a function satisfying the
Lipschitz condition, with uniformly bounded Lipschitz constants. 

For any bounded domain $\Omega\subset\RR^d$, we  denote by 
$\chi(\Omega)$ its \emph{chunkiness parameter} \cite{BrennerScott08}, the quotient of the diameter of $\Omega$ to the maximum radius
of a ball $B\subset\Omega$ such that $\Omega$ is star-shaped with respect to $B$. Clearly, $\chi(\Omega)<\infty$ if and only if
$\Omega$ is bounded and star-shaped with respect to a ball.
It is known that such domains satisfy the Lipschitz
condition, see \cite[Section 4.3]{Burenkov98} 
or \cite[Section 1.3.2]{MazyaPoborchi97}. 

Since $\partial\Omega$ is compact, there is a finite system $\cL$ consisting of $P_\cL$ open sets in $\RR^d$ that cover 
$\partial\Omega$ such
that their intersections with $\partial\Omega$ are the graphs of continuous functions with Lipschitz constants not exceeding 
some number $M_\cL$. Such a system is not unique, and  a smaller $M_\cL$ may be obtained at the expense of increasing $P_\cL$.
We will call any such system $\cL$ a \emph{Lipschitz cover} of $\partial\Omega$. Apart from $P_\cL$ and $M_\cL$, 
another parameter of $\cL$ important for the extension constant is a number $r_\cL>0$ such that 
for each $x\in\partial\Omega$ the  ball with radius $r_\cL$ centered at $x$  is contained in one of the sets of $\cL$. 
It is easy to see that a positive $r_\cL$ exists for any Lipschitz cover. 

The Stein Extension Theorem \cite{Stein} shows that $\cE_s(\Omega)$ is bounded from the above by a constant depending only on
$d,s,P_\cL,M_\cL$ and $r_\cL$.

Since we will need bounds for $\cE_s(\Omega)$ for domains with diameter tending to zero, we will  use 
the parameter $\tau_\cL:=r_\cL/\diam(\Omega)$ in place of $r_\cL$.
The following estimate follows immediately from the results in \cite[Section 3.1.5]{MazyaPoborchi97}.

\begin{lemma}\label{extce} Let $\Omega\subset\RR^d$ be a bounded domain star-shaped with respect to a ball, and let $\cL$ be a Lipschitz cover
of its boundary. If $\chi(\Omega)\le\chi$, $P_\cL\le P$, $M_\cL\le M$ and $\tau_\cL\ge \tau>0$, then for any integer $s\ge0$,
$$
\cE_s(\Omega)\le Ch^{-d/2},\qquad h=\diam(\Omega),$$
where $C$ depends only on $d,s,\chi,P,M$ and $\tau$.
\end{lemma}

\subsection{Sampling inequalities}\label{sample}
Sampling inequalities bound a weaker Sobolev norm of a function $f$ in terms of its values on a finite set $X\subset\Omega$ and a stronger Sobolev norm of $f$,
see the survey \cite{RSZ08} and references therein. In Lemma~\ref{sampling} we follow a standard way of proving sampling inequalities, as e.g.\ in  \cite{NWW05,WR05},
but obtain a local version that makes use of the diameter of the domain $\Omega$ as the discretization parameter $h$ instead of the fill distance of $X$. 
Therefore our estimates \eqref{sam2}, \eqref{saminf} depend on a polynomial Lebesgue constant, which is usually
avoided in global estimates by requiring that $X$ is sufficiently dense in $\Omega$. 
We refer to \cite[Section 4]{DavySchaback16} for a demonstration that this requirement is too restrictive for the
setting of local error bounds to which we will apply the sampling inequalities in Section~\ref{ldm}. A different local sampling inequality in terms of a growth function is
given in \cite[Corollary 13]{DavySchaback16}.

%
%
%
%
%

For any finite  $Y\subset\RR^d$ we consider the \emph{Lebesgue function}
$$
\lambda_{r}(y,Y)=\sup\{|p(y)|:\|p|_Y\|_\infty\le 1\text{ for all }p\in \Pi^d_r\},
\quad y\in\RR^d,\quad r\in\NN,$$
where $\Pi^d_r$, $r\in\NN$, 
denotes the space of $d$-variate polynomials of total order at most $r$ (or degree less than $r$).
Note that $\lambda_{r}(y,Y)<\infty$ for all $y\in\RR^d$ if and only if $Y$ is a \emph{$r$-determining set}, that is
$p\in \Pi^d_r$ and $p|_Y=0$ imply $p\equiv 0$. For any set $\Omega\subset\RR^d$, the \emph{Lebesgue constant} is given by
$$
\lambda_{r}(\Omega,Y)=\sup_{y\in \Omega}\lambda_{r}(y,Y).$$
Note that the \emph{Lebesgue constant} is the reciprocal of the \emph{norming constant} \cite{JSW99}. On the other hand, the Lebesgue function is a special case of a 
\emph{growth function} \cite{Davy07,DavySchaback18}, and we rely in  the proof on a general duality theory for growth functions \cite[Theorem 9]{DavySchaback18}
that replaces the ``local polynomial reproduction'' arguments, compare \cite[Chapter 3]{Wendland}.

\begin{lemma}\label{sampling} Let $\Omega\subset\RR^d$ be a bounded domain with chunkiness parameter $\chi(\Omega)\le\chi$ for some
constant $\chi<\infty$,
let $X=\{x_1,\ldots,x_n\}\subset\Omega$ be an $r$-determining set for some integer $r>d/2$, and let $v\in H^r(\Omega)$. 
We set $h=\diam(\Omega)$.
Then for any integer $k$ with $0\le k<r$,
\begin{equation}\label{sam2}
|v|_{H^k(\Omega)}\le Ch^{-k}\lambda_r(\Omega,X)\big(h^{d/2}\|v|_X\|_\infty + h^{r}|v|_{H^r(\Omega)}\big),
\end{equation}
where $C$ depends only on $r,d$ and $\chi$. Moreover,  for any $\beta\in\ZZ_+^d$ with $|\beta|<r-d/2$,
\begin{equation}\label{saminf}
|\partial^\beta v(x)|\le Ch^{-|\beta|}\lambda_r(\Omega,X)\big(\|v|_X\|_\infty + h^{r-d/2}|v|_{H^r(\Omega)}\big),
\quad x\in\Omega,
\end{equation}
where $C$ depends only on $r,d$ and $\chi$. 
\end{lemma}

\begin{proof}
By Sobolev embedding we identify $v$ with an $m$ times  continuously differentiable function in $\Omega$,
where $m$ is the largest integer such that $m<r-d/2$. 
By Proposition 4.3.2,
equation (4.1.18) and Lemma 4.3.8
in \cite{BrennerScott08} there is a polynomial $p\in \Pi^d_r$ such that
\begin{align}
\label{BHinf}
|\partial^\beta v(x)-\partial^\beta p(x)|&\le C_1h^{r-|\beta|-d/2}|v|_{H^r(\Omega)},\quad x\in\Omega,\quad |\beta|< r-d/2,\\
\label{BH}
|v- p|_{H^k(\Omega)}&\le C_2h^{r-k}|v|_{H^r(\Omega)},\quad k\le r,
\end{align}
where both $C_1$ and $C_2$ depend only on $r,d$ and $\chi$. 

Let $\alpha\in\ZZ_+^d$ with $|\alpha|=k$. Since $X$ is $r$-determining, for any $x\in\Omega$ there exist weights 
$w_j\in\RR$ such that
$$
\partial^\alpha p(x)=\sum_{j=1}^n w_j p(x_j)\quad\text{ and }\quad
\sum_{j=1}^n |w_j|= \sup\big\{\partial^\alpha q(x):q\in \Pi^d_r,\;\|q|_X\|\le 1\big\},$$
see \cite[Theorem 9]{DavySchaback18}.
By Markov inequality (see e.g.~\cite[Proposition 2.2]{NWW05}),
$$
|\partial^\alpha q(x)|\le C_3h^{-k}\|q\|_{L^\infty(\Omega)},\qquad q\in \Pi^d_r,$$ 
where $C_3$ depends only on $r,d$ and $\chi$. Hence
$$
\sum_{j=1}^n |w_j|\le C_3h^{-k}\lambda_r(\Omega,X).$$
By \eqref{BHinf},
$$
|v(x_j)-p(x_j)|\le C_1h^{r-d/2}|v|_{H^r(\Omega)},\quad j=1,\ldots,n.$$
Hence
\begin{align}
\begin{split}\label{pest}
|\partial^\alpha p(x)| &=  \Big|\sum_{j=1}^n w_j v(x_j)+\sum_{j=1}^n w_j \big(p(x_j)-v(x_j)\big)\Big|\\
&\le C_3h^{-k}\lambda_r(\Omega,X)\big(\|v|_X\|_\infty + C_1 h^{r-d/2}|v|_{H^r(\Omega)}\big),
\end{split}
\end{align}
which implies
$$
\|\partial^\alpha p\|_{L^2(\Omega)}
\le C_4C_3h^{-k+d/2}\lambda_r(\Omega,X)\big(\|v|_X\|_\infty + C_1h^{r-d/2}|v|_{H^r(\Omega)}\big),$$
where $C_4$ depends only on $d$.
By \eqref{BH},
$$
\|\partial^\alpha v-\partial^\alpha p\|_{L^2(\Omega)}\le C_2h^{r-k}|v|_{H^r(\Omega)},$$
and \eqref{sam2} follows since 
$\|\partial^\alpha v\|_{L^2(\Omega)}
\le \|\partial^\alpha v-\partial^\alpha p\|_{L^2(\Omega)}+\|\partial^\alpha p\|_{L^2(\Omega)}$
and $\lambda_r(\Omega,X)\ge1$. Similarly, \eqref{saminf} follows directly from \eqref{BHinf} and \eqref{pest}.
\end{proof}

\subsection{Local differentiation matrices}\label{ldm}
Let $L$ be a differential operator of order $m$ on $\cM$, $K$ the reproducing kernel of a Hilbert space $H_K(\cM)$
embedded into $C^m(\cM)$, and $X=\{x_1,\ldots,x_n\}\subset\cM$. For any $u\in C^m(\cM)$, an approximation of $Lu$ can be
obtained by applying $L$ to the kernel interpolant \eqref{kint},
\begin{align*}
\sigma&=\sum_{j=1}^n c_j K(\cdot,x_j),\quad \sigma(x_j)=u(x_j),\quad j=1,\ldots,n,\\
Lu&\approx L\sigma=\sum_{j=1}^n c_j K_L(\cdot,x_j).
\end{align*}
In particular, the approximation  $L\sigma|_X$ of $Lu|_X$ can be computed with the help of the \emph{differentiation matrix}
$W_X$, 
\begin{equation}\label{difm}
L\sigma|_X=W_X u|_X ,\quad\text{where}\quad W_X:=K_L|_{X}\,(K|_{X})^{-1}.
\end{equation}

If $U$ is an open set in $\cM$ and $\varphi:U\to\RR^d$ is a homeomorphism from $U$ to 
$U^\varphi:=\varphi(U)\subset\RR^d$, then we denote by $L^\varphi$ the differential operator $L$ in the local coordinates defined 
by $\varphi$, that is 
$$
L^\varphi v:=L(v\circ\varphi)\circ\varphi^{-1}$$
 for any sufficiently smooth function $v$ on 
$U^\varphi$. Then the kernel $K^\varphi(x,y)=K(\varphi^{-1}(x),\varphi^{-1}(y))$ of the space
$H_K(U^\varphi)$ defined in \eqref{Hphi} obviously satisfies
\begin{equation}\label{L1Kphi}
L^\varphi_iK^\varphi=(L_iK)^\varphi,\quad i=1,2,
\end{equation}
where $L_1K$ and $L_2K$ denote the result of applying the operator $L$ to the first, respectively, second argument of $K$,
according to the notation introduced in Section~\ref{pdk}.
In particular,
$$
W_X=K^\varphi_L|_{X^\varphi}\,(K^\varphi|_{X^\varphi})^{-1},$$
where $K^\varphi_L:=L^\varphi_1K^\varphi$ and $X^\varphi:=\varphi(X)$.

\begin{lemma}\label{lde} Let $L$ be a differential operator on $\cM$ of order $m$ with smooth coefficients, and 
let $K$ be a reproducing kernel for $H^s(\cM)$, with an integer $s>m+d/2$. 
Furthermore, let $X=\{x_1,\ldots,x_n\}\subset U$ for an open set $U\subset \hat U^c_k\subset\cM$
for some $k\in\{1,\ldots,\hat n\}$ and $c>0$. For $\varphi=\hat\varphi_k$, assume that   $U^\varphi$
is star-shaped with respect to a ball, with $\chi(U^\varphi)\le\chi<\infty$, and 
$X^\varphi$ is an $r$-determining set for some $r>m+d/2$ such that $K(\cdot,x)\in H^r(\cM)$ for all $x\in\cM$. 
Then for any $u\in H^r(\cM)$,
\begin{equation}\label{lde1}
\|Lu|_X-W_X u|_X\|_\infty\le 
C_1\lambda_r(U^\varphi,X^\varphi)h^{r-m-d/2}|u^\varphi-\sigma^\varphi|_{H^r(U^\varphi)},
\end{equation}
where $h=\diam(U^\varphi)$, $u^\varphi:=u\circ\varphi^{-1}$, $\sigma^\varphi:=\sigma\circ\varphi^{-1}$ and 
$C_1$ depends only on $r,d,\chi$ and $L$. 
Moreover, if  $X^\varphi$ is an $s$-determining set and $u\in H^s(\cM)$, then
\begin{align}\label{lde2}
\|Lu|_X-W_X u|_X\|_\infty&\le 
C_2\lambda_s(U^\varphi,X^\varphi)h^{s-m-d/2}\|u\|_{K,U}\\
\label{lde3}
&\le 
C_3\lambda_s(U^\varphi,X^\varphi)\cE_s(U^\varphi)h^{s-m-d/2}\|u^\varphi\|_{H^s(U^\varphi)},
\end{align}
where $C_2,C_3$ depend in addition on $K$, $\hat\cA$, $\hat\Gamma$ and $c$.
\end{lemma}

\begin{proof} By Sobolev embedding, the condition $r>m+d/2$ ensures that $u\in C^m(\cM)$, hence $Lu$ is well defined as a continuous function on $\cM$
as soon as $u\in H^r(\cM)$.
Since $U^\varphi$ is a Lipschitz domain, it follows that $\cE_s(U^\varphi)<\infty$, and Lemma~\ref{SobKm} shows that
$K^\varphi$ is a reproducing kernel for $H^s(U^\varphi)$. Since $s>m+d/2$, $H^s(U^\varphi)$ is embedded
into $C^m(\cM)$. Moreover, $\sigma^\varphi\in H^{r}(U^\varphi)$ as a linear combination of the functions $K^\varphi(\cdot,x_j)$, $j=1,\ldots,n$.

In view of \eqref{difm} and \eqref{L1Kphi}, the $i$-th component of the vector $Lu|_X-W_X u|_X$
has the form 
$$
Lu(x_i)-L\sigma(x_i)=L^\varphi u^\varphi(x^\varphi_i)-L^\varphi \sigma^\varphi(x^\varphi_i),$$
where  $x^\varphi_i=\varphi(x_i)$. Since $(u^\varphi-\sigma^\varphi)|_{X^\varphi}=0$,
the sampling inequality \eqref{saminf} implies
$$
|L^\varphi u^\varphi(x^\varphi_i)-L^\varphi \sigma^\varphi(x^\varphi_i)|
\le C\lambda_r(U^\varphi,X^\varphi)h^{r-m-d/2}|u^\varphi-\sigma^\varphi|_{H^r(U^\varphi)},$$
where $C$ depends only on $r,d,\chi$ and the constants in \eqref{aalpha}, which implies \eqref{lde1}.

In the case $r=s$ the condition that $K(\cdot,x)\in H^s(\cM)$ for all $x\in\cM$ is satisfied since $K$ is  a reproducing kernel for $H^s(\cM)$, and
we have by Lemma~\ref{SobKm},
$$
|u^\varphi-\sigma^\varphi|_{H^s(U^\varphi)}\le \|u^\varphi-\sigma^\varphi\|_{H^s(U^\varphi)}\le 
C\|u-\sigma\|_{K,U},$$
where $C$ depends only on $K$, $\hat\cA$, $\hat\Gamma$ and $c$.
Then the minimum norm property \eqref{minnorm} of the kernel interpolant implies 
$\|u-\sigma\|_{K,U}\le \|u\|_{K,U}$, and \eqref{lde2} follows.

Finally, by Lemma~\ref{SobKm}, $\|u\|_{K,U}\le C\cE_s(U^\varphi)\|u^\varphi\|_{H^s(U^\varphi)}$,
 where $C$ depends only on $K$, $\hat\cA$, $\hat\Gamma$, $c$ and $s$, which implies \eqref{lde3}.
\end{proof}

Different types of local error bounds for kernel based numerical differentiation can be found in 
\cite{DavySchaback16,DavySchaback19}.  

Note that for an operator $L$ and kernel $K$ satisfying the hypotheses of Lemma~\ref{KL}, the estimate \eqref{lde1} holds, 
thanks to Lemma~\ref{smkerL}, for all $r$ satisfying $2\kappa + d/2 < r< 2s-d/2$ such that $r-s$ is a multiple of $2\kappa$. 
Hence, it holds in this case for all integer $r$ with $2\kappa + d/2 < r\le s+2\kappa$.

\section{Error bounds}\label{estimates}
 
The discrete solution $\hu$ is well-defined by \eqref{leastsq}  for any node set $X$ with subsets $X_\ell$ such that $\bigcup_{\ell =1}^m X_\ell =X$
as soon as $L$ 
is positive and self-adjoint with respect to the inner product defined by the reproducing kernel
$K$ for $H^s(\cM)$ with $s>\kappa+d/2$.
However, in order to provide an estimate of the error of  $\hu$,
 \begin{equation}\label{der}
\|\hat u-u|_{X}\|_\infty,
\end{equation}
we will make a number of additional assumptions. 
In particular, we will need to assume that
\begin{equation}\label{As}
\text{$s$ is an integer satisfying $s>2\kappa+d/2$,}
\end{equation}
rather than allowing any real $s>\kappa+d/2$. This restriction comes from an application of Lemma~\ref{lde} in the proof of
Theorem~\ref{error_bound}.

As in Section~\ref{Sob}, we denote by $\|\cdot\|_{H^s(\cM)}$ the norm of $H^s(\cM)$ associated with a fixed finite smooth atlas
$\hat\cA=\{(\hat U_k,\hat\varphi_k)\}_{k=1}^{\hat n}$ and a smooth
partition of unity $\hat\Gamma=\{\hat\gamma_k\}_{k=1}^{\hat n}$ subordinate to it.

We assume that the sets $X_\ell$ of the system $\cX$ satisfy
\begin{equation}\label{XA}
X_\ell=X\cap U_\ell, \quad \ell =1,\ldots,m,
\end{equation}
for a finite smooth atlas  $\cA=\{(U_\ell,\varphi_\ell)\}_{\ell=1}^m$ of $\cM$, such that:

\renewcommand{\labelenumi}{(A\arabic{enumi})}
\newcommand{\Aref}[1]{(A\ref{#1})}

\begin{enumerate}


\item\label{A1} $\cA$ is $c$-admissible with respect to $\hat\cA$ and 
$\hat\Gamma$, for some $c=c_\cA>0$. 

\item\label{A2} Each $U_\ell^\varphi:= \varphi_\ell(U_\ell)$ is bounded and star-shaped with respect to a ball 
$B_\ell\subset\RR^d$. Thus, the following quantity is finite:
$$
\chi_\cA:=\max_{\ell=1,\ldots,n}\chi(U_\ell^\varphi),$$
where $\chi(\Omega)$ denotes the chunkiness parameter of $\Omega$, as defined in Section~\ref{extc}.

\item\label{A3} 
%
For all $\ell$ and $p$ such that 
$U_p\cap U_\ell\ne\emptyset$, the set
$\varphi_\ell(X_p\cap X_\ell)$ is an $(s+\kappa)$-determining set. 
In other words, $\lambda_{s+\kappa}(\cA,X)<\infty$, where 
\begin{align*}
\lambda_{r}(\cA,X):=
\max_{1\le\ell\le m}\max_{i\in I_\ell}
\lambda_{r}(U_{\ell}^\varphi,\varphi_\ell(X_i\cap X_\ell)),
\quad I_\ell:=\{i: U_i\cap U_\ell\ne\emptyset\},\quad r\in\ZZ_+.
\end{align*}


\end{enumerate}
\renewcommand{\labelenumi}{\arabic{enumi}.}
\newcommand{\chain}{A\ref{A3}$'$}
Condition \Aref{A3} may be weakened, see Remark~\ref{chains} below. 

Note that these conditions can be taken into account when selecting the sets $X_\ell$. For example,
we may choose $U_\ell=\hat\varphi_k^{-1}(B_\ell)$ for suitable balls $B_\ell\subset \hat\varphi_k(\hat U_k^c)$ and
$1\le k\le \hat n$, 
such that $U_\ell$
and $U_p$ overlap significantly if $U_\ell\cap U_p\ne \emptyset$,
and define the subsets $X_\ell$ of $X$ as $X_\ell=X\cap U_\ell$. If we first select the system $\cX$, then we need to make
sure that the intersections  $X_\ell\cap X_p$ are sufficiently big when non-empty to ensure the existence of a suitable atlas
$\cA$. In this case Condition (\chain) of Remark~\ref{chains}   is significantly easier to fulfill than \Aref{A3}.



In addition to the parameters $c_\cA>0$, $\chi_\cA<\infty$ and $\lambda_{s+\kappa}(\cA,X)<\infty$ of the assumptions
\Aref{A1}--\Aref{A3}, we now define several further quantities needed in the estimate. Since by \eqref{XA} the sets $X_\ell$ are 
determined by $X$ and $\cA$, we do not indicate in the notation when these quantities depend on $\{X_\ell\}_{\ell=1}^m$.
\begin{itemize}
\item
As the main parameter, against which the error will be measured, we will use
$$
h_\cA=\max_{1\le \ell\le m}h_\ell,$$
where $h_\ell$ is the diameter of the smallest open ball in $\RR^d$ containing $U_\ell^\varphi$.

\item
Furthermore, we set
$$
\nu_{X,\cA}:=\max_{1\le \ell\le m}n_\ell,$$
where $n_\ell=\# X_\ell$ as in Section~\ref{method}.

\item
We denote by $\delta_\ell$ 
the \emph{separation distance} of the set $X_\ell$ in $U_\ell$, 
that is, the largest $\delta>0$ such that 
the open balls of radius $\delta$ centered at $\varphi_\ell(x_j)$ for all $j\in J_\ell$ 
are pairwise disjoint and contained in $U_\ell^\varphi$, 
\begin{equation}\label{sep}
\delta_\ell:=\min\big\{\!\dist(X_\ell^\varphi,\partial U_\ell^\varphi),\,\tfrac12\min_{j\in J_\ell}
\dist\big(\varphi_\ell(x_j),X_\ell^\varphi\setminus\{\varphi_\ell(x_j)\}\big)\big\},\; \ell=1,\ldots,n,
\end{equation}
where $X_\ell^\varphi:= \varphi_\ell(X_\ell)$.

\item
To measure the \emph{quasi-uniformity} of $\cA$ and $X$, we define
$$
q_{\cA}:=\max_{1\le i,j\le m}\frac{h_i}{h_j}=\max_{1\le \ell\le m}\frac{h_\cA}{h_\ell},\qquad 
q_{X,\cA}:=\max_{1\le \ell\le m}\frac{h_\cA}{2n_\ell^{1/d}\delta_\ell}.$$
Note that 
\begin{equation}\label{qA}
q_{\cA}< q_{X,\cA}
\end{equation}
since $n_\ell\delta_\ell^d<\big(\frac{h_\ell}{2}\big)^d$ for all $\ell$, which is clear from the comparison of the
volume of the union of the balls in the definition of $\delta_\ell$ with the volume of a ball of diameter $h_\ell$ containing
$U_\ell^\varphi$. 

\item
It follows from \Aref{A2} that the domains $U_\ell^\varphi$ satisfy the Lipschitz condition. For each $\ell$, we
choose a Lipschitz cover $\cL$ of $\partial U_\ell^\varphi$ and consider $P_\ell=P_\cL$, $M_\ell=M_\cL$, 
$\tau_\ell=\tau_\cL$, as defined in Section~\ref{extc}. Let
$$
P_\cA:=\max_{\ell=1,\ldots,n}P_\ell,\qquad
M_\cA:=\max_{\ell=1,\ldots,n}M_\ell,\qquad
\tau_\cA:=\min_{\ell=1,\ldots,n}\tau_\ell,$$
and
$$
\cE_r(\cA):=\max_{\ell=1,\ldots,n}\cE_r(U_\ell^\varphi),\quad r\in\ZZ_+.$$

%

\item
Finally, we choose a smooth partition of unity $\Gamma=\{\gamma_\ell\}_{\ell=1}^n$ subordinate to $\cA$  such that 
\begin{equation}\label{pum}
\gamma_\ell\in C^\infty(\cM),\quad\gamma_\ell\ge0,\quad 
\supp \gamma_\ell\subset U_\ell,\quad
\sum_{\ell=1}^n\gamma_\ell=1,
\end{equation}
and set
\begin{equation}\label{eta}
\eta_r(\cA):=\max_{1\le\ell\le m}\max_{|\alpha|\le r}\|\partial^\alpha g_\ell\|_{L^\infty(\RR^d)},\qquad r\in\ZZ_+,
\end{equation}
where $g_\ell:\RR^d\to\RR$ is defined by
$$
g_\ell(y):= 
\begin{cases}
(\gamma_\ell\circ\varphi_\ell^{-1})(y_\ell+h_\ell y), & \text{if }y_\ell+h_\ell y\in U_{\ell}^\varphi,\\
0, &\text{otherwise,}
\end{cases}$$
with $y_\ell\in \RR^d$ being the center of the ball $B_\ell$ in \Aref{A2}.

\end{itemize}

Although a subordinate partition of unity $\Gamma$ always exists and corresponding $\eta_r(\cA)$  is finite, we provide an 
example of how a suitable $\Gamma$ may be constructed and $\eta_r(\cA)$ estimated
under an easy to check condition on $\cA$.

\begin{example}\label{pume}\rm
According to \Aref{A2}, each $U_{\ell}^\varphi$ contains a ball $B_\ell$ with radius
$\rho_\ell$ and center $y_\ell$, such that $h_\ell\le \chi_\cA \rho_\ell$. We assume that
there exists a positive $\alpha<1$, such that
\begin{equation}\label{alpha}
\text{for  each $x\in\cM$ there is an $\ell$ with $x\in U_\ell$ and
$\|\varphi_\ell(x)-y_\ell\|_2\le \alpha\rho_\ell$.}
\end{equation}
Let $\Psi:\RR^d\to\RR$ be an infinitely differentiable bump function supported on the unit ball of $\RR^d$, 
for example $\Psi(y)=\psi(\|y\|_2^2)$ for some $\psi\in C^\infty[0,\infty)$, with $\psi(0)=1$, $\psi^{(k)}(0)=0$, 
$k\in\NN$, $\psi(t)>0$ and $\psi'(t)\le0$ for $0\le t<1$, and $\psi(t)=0$ for $t\ge1$. 
We set
$$
\bar\gamma_\ell(x):= 
\begin{cases}
\Psi\big(\tfrac{\varphi_\ell(x)-y_\ell}{\rho_\ell} \big), & x\in U_{\ell},\\
0, & x\in\cM\setminus U_\ell,
\end{cases}$$
and
$$
\gamma_\ell(x)=\tfrac{\bar\gamma_\ell(x)}{\bar\gamma(x)},\quad \ell=1,\ldots,m,\qquad\text{with}\quad
\bar\gamma(x):=\sum_{i=1}^m \bar\gamma_i(x).$$
Then
$$
\bar\gamma(x)\ge \psi(\alpha^2)>0,\quad x\in\cM,$$
each function $g_\ell$ is supported in the ball of radius $\rho_\ell/h_\ell$ centered at the origin, 
and for $y$ in this ball,
%
$$
g_\ell(y)=\frac{\Psi(h_\ell y/\rho_\ell)}{\bar\gamma(\varphi_\ell^{-1}(y_\ell+h_\ell y))}.$$
Here
\begin{align*}
\bar\gamma(\varphi_\ell^{-1}(y_\ell+h_\ell y))
&=\sum_{i\in I_\ell(y)}\Psi\big(\tfrac{(\varphi_i\circ \varphi_\ell^{-1})(y_\ell+h_\ell y)-y_i}{\rho_i}\big),\\
I_\ell(y)&:=\{i: \varphi_\ell^{-1}(y_\ell+h_\ell y)\in U_i\}\subset\{i: U_i\cap U_\ell\ne\emptyset\}.
\end{align*}
Then by the Leibniz and the chain rules it is easy to see that $\eta_r(\cA)$ can be bounded above in terms of
(the upper bounds for)
$$
\chi_\cA,\quad  \mu_\cA,\quad q_{\cA}\quad\text{and}\quad 1/\psi(\alpha^2),$$
as well as the constant 
$$
\max_{i,j=1,\ldots,\hat n}\max_{|\alpha|\le r}\|\partial^\alpha (\hat\varphi_i\circ \hat\varphi_j^{-1})\|_{L^\infty(\hat\varphi_j(\hat U_j))}$$
that depends only on the atlas $\hat\cA$ that defines \eqref{Snorm}.
Note that more general domains  e.g.\ unions of ellipsoids 
may be used instead of balls as supports of the partition of unity functions $\gamma_\ell$ 
if we want to estimate the smallest possible constant $\eta_r(\cA)$ 
for a given choice of  $U_\ell$ and $X_\ell=X\cap U_\ell$ and given $r$. The size of  $\eta_r(\cA)$ for $r=s+\kappa$ 
influences the constants in the error bounds below.
\end{example}

We stress that in contrast to the partition of unity methods \cite{BabuskaMelenk97,GriebelSchweitzer02,LSH17}, 
$\Gamma$ is not used in the computation of the discrete solution $\hu\approx u|_X$, and we only need  
to choose a partitions of unity  in order to prove the error bounds.

\begin{theorem}\label{error_bound}
Let $L$ be an elliptic differential operator  of order $2\kappa$ with smooth coefficients  
and trivial null space, and let $K$ be a reproducing kernel for a Sobolev space $H^s(\cM)$ with an integer
$s$ satisfying $s>2\kappa+d/2$,
such that $L$  is positive and self-adjoint with respect to the inner product defined by $K$. 
Assume that a set of nodes $X$ and an atlas $\cA=\{(U_\ell,\varphi_\ell)\}_{\ell=1}^m$ satisfy 
\Aref{A1}--\Aref{A3}. For any $f\in H^{s-\kappa}(\cM)$, let $u\in H^{s+\kappa}(\cM)$ be the 
solution of $Lu=f$, and let $\hat u$ be the discrete approximate solution  determined by \eqref{leastsq}. 
Then there exists a function $\tilde u\in H^{s+2\kappa}(\cM)$ such that $\tilde u|_X=\hu$ and for all
$0\le r<s-\kappa$,
\begin{equation}\label{Hb}
\|u-\tilde u\|_{H^{2\kappa+r}(\cM)}\le
C_1h_\cA^{s-2\kappa-d-r}\|u\|_{H^s(\cM)}
+C_2h_\cA^{s-\kappa-d/2-r}\|u\|_{H^{s+\kappa}(\cM)},
\end{equation}
where $C_1,C_2<\infty$ are independent of $f,u$ and $h_\cA$.
If $\mu_\cA\le\mu$, $\lambda_{s+\kappa}(\cA,X)\le\lambda$, $\chi_\cA\le\chi$, $\nu_{X,\cA}\le\nu$,
$q_{X,\cA}\le q$, $\eta_{s+\kappa}(\cA)\le\eta$, $P_\cA\le P$, $M_\cA\le M$, $\tau_\cA\ge\tau$, $c_\cA\ge c$, 
for some positive real $\mu,\lambda,\chi,\nu,q,\eta,P,M,\tau,c$,
then \eqref{Hb} holds with $C_1,C_2$
depending only on $\mu,\lambda,\chi,\nu,q,\eta,P,M,\tau,c$, in addition to $\cM,L,K,\hat\cA$ and $\hat\Gamma$.
Moreover, 
\begin{equation}\label{infb}
\|u|_{X}-\hat u\|_\infty\le\|u-\tilde u\|_{C(\cM)}\le C_1h_\cA^{s-2\kappa-d-r_0}\|u\|_{H^s(\cM)}
+C_2h_\cA^{s-\kappa-d/2-r_0}\|u\|_{H^{s+\kappa}(\cM)},
\end{equation}
where $r_0=\max\{0,\lfloor \frac{d}{2}\rfloor-2\kappa+1\}$, with the same properties of  $C_1,C_2$.
\end{theorem}


\begin{proof}
Recall that we assume that $\mu_\cA\le \mu$,  $\lambda_{s+\kappa}(\cA,X)\le\lambda$, 
$c_\cA\ge c$, $\chi_\cA\le\chi$, $\nu_{X,\cA}\le\nu$,
$q_{X,\cA}\le q$, $\eta_{s+\kappa}(\cA)\le\eta$. Then also $\lambda_{r}(\cA,X)\le\lambda$ for all $r\le s+\kappa$, in particular for $r=s$ and $r=s-\kappa$. 
The constants denoted by $C,C_1,C_2,\ldots$ in
the proof  will be independent of $f,u,h_\cA$ and the Lipschitz cover parameters  $P_\cA,M_\cA,\tau_\cA$. 
They in general depend on $\cM,L,K$, the atlas
$\hat\cA$ and the partition of unity $\hat\Gamma$ that we use to fix the Sobolev norms, as well as on
$\mu,\lambda,\chi,\nu,q,\eta,c$. The constants in \eqref{Hb} and \eqref{infb} will also depend on $P,M,\tau$ via the final
application of Lemma~\ref{extce} to \eqref{errorE}.
We assume without loss of generality that $h_\cA\le1$. 

Since $s>2\kappa+d/2$, we obtain by \eqref{lde3} and \Aref{A1}--\Aref{A3},
\begin{equation}\label{ldep}
\|\Wl u|_{X_\ell} - Lu|_{X_\ell}\|_\infty\le C\cE_{s}(U^\varphi_\ell)h_\cA^{s-2\kappa-d/2}\|u\circ\varphi_\ell^{-1}\|_{H^s(U^\varphi_\ell)},
\quad \ell=1,\ldots,m.
\end{equation}
Since, in view of \eqref{leastsq} and \eqref{ods2},
$$
\|W\!M\hat u-Mf|_X\|_2^2\le\|W\!Mu|_X-Mf|_X\|_2^2
=\sum_{\ell=1}^{m}\|\Wl u|_{X_\ell} - Lu|_{X_\ell}\|_2^2,$$
we obtain by \eqref{ldep} and Lemma~\ref{admis}, 
\begin{equation}\label{eps}
\eps_1:=\|W\!M\hat u-Mf|_X\|_2\le C\cE_{s}(\cA)h_\cA^{s-2\kappa-d/2}\|u\|_{H^s(\cM)}.
\end{equation}

We will construct a function $\tilde u\in H^{s+\kappa}(\cM)$  such that 
$\tilde u|_{X}=\hat u$, and estimate its distance to $u$ in Sobolev spaces of lower order as in \eqref{Hb}. We define local interpolants $\tilde u_\ell$
on the sets $X_\ell$ and then combine them together with the help of a partition of unity. 

For each $\ell=1,\ldots,m$, let
$$
\tilde u_\ell(x)=\sum_{j\in J_\ell}c_j K(x,x_j),$$
where the coefficients $c_j$ are uniquely determined by the interpolation conditions
%
%
%
\begin{equation}\label{uellint}
\tilde u_\ell|_{X_\ell}=\hat u|_{J_\ell}.
\end{equation}
By \eqref{Wl},
$L\tilde u_\ell|_{X_\ell}=W_\ell \tilde u_\ell|_{X_\ell}$. Hence
\begin{equation}\label{giell} 
L \tilde u_\ell|_{X_\ell}= W_\ell \hat u|_{J_\ell}.
\end{equation}
Moreover,  by Lemma~\ref{smkerL}, $\tilde u_\ell\in H^{s+2\kappa}(\cM)$.
By \eqref{reg} and Lemmas~\ref{SobK} and \ref{KL},
\begin{equation}\label{uelln} 
\|\tilde u_\ell\|_{H^{s+\kappa}(\cM)}\le C\|L\tilde u_\ell\|_{K_L}.
\end{equation}
It follows from  \eqref{giell} and \eqref{ods2}
 that
\begin{equation}\label{eps4}
\sum_{\ell=1}^{m}\|L \tilde u_\ell|_{X_\ell}-Lu|_{X_\ell}\|_2^2=
\sum_{\ell=1}^{m}\|W_\ell \hat u|_{J_\ell}-Lu|_{X_\ell}\|_2^2=\|W\!M\hat u-Mf|_X\|_2^2=\eps_1^2.
\end{equation}

Our next goal is to estimate the  $K_L$-norm of $L\tilde u_\ell$. Note that due to \eqref{srestr}, 
$\|L\tilde u_\ell\|_{K_L}=\|L\tilde u_\ell\|_{K_L,U_\ell}$ since $X_\ell\subset U_\ell$. Let 
$$
u_\ell\in\spann\{K(\cdot,x_j):\,j\in J_\ell\}$$
be the kernel sum determined by the interpolation conditions
$$
L u_\ell|_{X_\ell}=L u|_{X_\ell}.$$ 
Then $L u_\ell$ is the $K_L$-kernel interpolant of $f=Lu$. 
By Lemma~\ref{KL}, $K_L$ is a reproducing kernel for $H^{s-\kappa}(\cM)$,
and by Lemma~\ref{SobKm} and \Aref{A2}, $H_{K_L}(U_\ell)=H^{s-\kappa}(U_\ell^\varphi)$. Hence,  by \eqref{srestr} and the minimum norm property
\eqref{minnorm}, since $X_\ell\subset U_\ell$, 
$$
\|L u_\ell\|_{K_L}=\|L u_\ell\|_{K_L,{U_\ell}}\le \|f\|_{K_L,{U_\ell}},$$
and by Lemma~\ref{SobKm},
$$
\|f\|_{K_L,{U_\ell}}\le C\cE_{s-\kappa}(U^\varphi_\ell)\|f \circ\varphi_\ell^{-1}\|_{H^{s-\kappa}(U_\ell^\varphi)}.$$
It follows by Lemma~\ref{admis} that
\begin{equation}\label{b1}
\sum_{\ell=1}^m\|L u_\ell\|_{K_L}^2\le C\cE^2_{s-\kappa}(\cA)\|f\|_{H^{s-\kappa}(\cM)}^2.
\end{equation}
Consider the functions $v_\ell=L\tilde u_\ell - Lu_\ell$, $\ell=1,\ldots,m$, that satisfy
$$
\sum_{\ell=1}^{m}\|v_\ell|_{X_\ell}\|_2^2\le\eps_1^2,$$
thanks to \eqref{eps4}. Let
$$
\psi_\ell=\sum_{j\in J_\ell}v_\ell(x_j)\,\Psi_{j,\delta_\ell}^\ell,
\qquad \Psi_{j,\delta}^\ell(x):=
\begin{cases}\Psi\big(\tfrac{\varphi_\ell(x)-\varphi_\ell(x_j)}{\delta}\big),& x\in U_\ell,\\
0, & x\in\cM\setminus U_\ell,
\end{cases}$$
where $\Psi(y)=\psi(\|y\|_2^2)$ is a radially symmetric bump function in $C^\infty(\RR^d)$ supported on the unit ball, 
as in Example~\ref{pume}, 
and $\delta_\ell$ is defined in \eqref{sep}. 
Then the supports of the functions
$\Psi_{j,\delta_\ell}^\ell$, $j\in  J_\ell$, are pairwise disjoint and hence
$$
\psi_\ell|_{X_\ell}=v_\ell|_{X_\ell}$$
and 
$$
|\psi_\ell\circ \varphi_\ell^{-1}|_{H^r(U_\ell^\varphi)}=\delta_\ell^{-r+d/2}|\Psi|_{H^r(\RR^d)}
\|v_\ell|_{X_\ell}\|_2,\qquad r\in\ZZ_+.$$
Since $\delta_\ell^{-1}<2\nu^{1/d}qh_\cA^{-1}$ and $\delta_\ell<h_\cA\le1$, we obtain
\begin{align*}
\sum_{\ell=1}^{m}\|\psi_\ell\circ \varphi_\ell^{-1}\|_{H^{s-\kappa}(U_\ell^\varphi)}^2
\le \sum_{r=0}^{s-\kappa}|\Psi|_{H^{r}(\RR^d)}^2\sum_{\ell=1}^{m}\delta_\ell^{-2r+d}\|v_\ell|_{X_\ell}\|_2^2
\le C\eps_1^2\,\|\Psi\|_{H^{s-\kappa}(\RR^d)}^2\,h_\cA^{-2(s-\kappa)+d}. 
\end{align*}
By \eqref{srestr}, the minimum norm property \eqref{minnorm} of $v_\ell$ as the $K_L$-kernel interpolant of 
$\psi_\ell\in H_{K_L}(U_\ell)$,  and by Lemma~\ref{SobKm},
$$
\|v_\ell\|_{K_L}=\|v_\ell\|_{K_L,U_\ell}\le \|\psi_\ell\|_{K_L,U_\ell}
\le C\cE_{s-\kappa}(U^\varphi_\ell)\|\psi_\ell\circ \varphi_\ell^{-1}\|_{H^{s-\kappa}(U_\ell^\varphi)},$$
and it follows from the above that
$$ 
\sum_{\ell=1}^{m}\|v_\ell\|_{K_L}^2
\le C\cE^2_{s-\kappa}(\cA)\eps_1^2h_\cA^{-2(s-\kappa)+d}.
$$ 
Combining this with \eqref{b1}, we obtain for $L\tilde u_\ell=Lu_\ell+v_\ell,$
\begin{equation}\label{b3}
\Big(\sum_{\ell=1}^{m}\|L \tilde u_\ell\|_{K_L}^2\Big)^{1/2}
\le \cE_{s-\kappa}(\cA)\big(C_1\|f\|_{H^{s-\kappa}(\cM)}+C_2\eps_1h_\cA^{-s+\kappa+d/2} \big).
\end{equation}



We set
$$
\tilde u :=\sum_{\ell=1}^m \gamma_\ell \tilde u_\ell,$$
where $\{\gamma_\ell\}_{\ell=1}^m$ is the smooth partition of unity \eqref{pum} subordinate to $\cA$.
Note that $\tilde u\in H^{s+2\kappa}(\cM)$, which follows from the same property of $\tilde u_\ell$.
It follows from \eqref{uellint} and \eqref{pum} that
\begin{equation}\label{uint}
\tilde u|_{X} = \hat u.
\end{equation}

We now use \eqref{b3} and the properties of the partition of unity
in order to obtain an estimate for $\|L\tilde u\|_{H^{s-\kappa}(\cM)}$. 
By Lemma~\ref{admis},
$$
\|L\tilde u\|_{H^{s-\kappa}(\cM)}^2
\le C\sum_{\ell=1}^m\|L\tilde u\circ \varphi_\ell^{-1}\|_{H^{s-\kappa}(U_\ell^\varphi)}^2.$$
For each $\ell=1,\ldots,m$,
$$
\|L\tilde u\circ \varphi_\ell^{-1}\|_{H^{s-\kappa}(U_\ell^\varphi)}
\le \|L\tilde u_\ell\circ \varphi_\ell^{-1}\|_{H^{s-\kappa}(U_\ell^\varphi)}+
\|L(\tilde u-\tilde u_\ell)\circ \varphi_\ell^{-1}\|_{H^{s-\kappa}(U_\ell^\varphi)}.$$
By Lemma~\ref{SobKm} and \eqref{srestr},
$$
\|L\tilde u_\ell\circ \varphi_\ell^{-1}\|_{H^{s-\kappa}(U_\ell^\varphi)}\le C \|L\tilde u_\ell\|_{K_L},$$
and we will be able to estimate the terms of this type by applying \eqref{b3}.
It follows from \eqref{aalpha} and the Leibniz rule that
$$
\|L(\tilde u-\tilde u_\ell)\circ \varphi_\ell^{-1}\|_{H^{s-\kappa}(U_\ell^\varphi)}
\le C \|(\tilde u-\tilde u_\ell)\circ \varphi_\ell^{-1}\|_{H^{s+\kappa}(U_\ell^\varphi)}.$$
By the definition of $\tilde u$,
$$
\tilde u-\tilde u_\ell
=\sum_{p\in I_\ell\setminus\{\ell\}}
\gamma_p(\tilde u_p -\tilde u_{\ell})
\quad\text{in $U_\ell$},\qquad I_\ell=\{i: U_i\cap U_\ell\ne\emptyset\}.$$
Applying the Leibniz rule again, we obtain
for  $|\alpha|\le s+\kappa$,
$$
\partial^\alpha\big( (\tilde u - \tilde u_{\ell})\circ\varphi_\ell^{-1}\big)
=\sum_{p\in I_\ell\setminus\{\ell\}}
\sum_{\beta\le\alpha}{\alpha\choose\beta}
\partial^{\alpha-\beta}(\gamma_p\circ \varphi_\ell^{-1})
\,\partial^{\beta}\big((\tilde u_p -\tilde u_{\ell})\circ \varphi_\ell^{-1}\big)\quad\text{in $U_\ell^\varphi$}.$$
By \eqref{eta},
$$
\|\partial^{\alpha-\beta}(\gamma_p\circ \varphi_\ell^{-1})\|_{L^\infty(U_\ell^\varphi)}
=h_\ell^{|\beta-\alpha|}\|\partial^{\alpha-\beta}g_\ell\|_{L^\infty(\RR^d)}
\le h_\ell^{|\beta-\alpha|}\eta_{s+\kappa}(\cA)\le\eta h_\ell^{|\beta-\alpha|}.$$
Hence
$$
\|\partial^\alpha\big( (\tilde u - \tilde u_{\ell})\circ\varphi_\ell^{-1}\big)\|_{L^2(U_\ell^\varphi)}
\le C\sum_{p\in I_\ell\setminus\{\ell\}}
\sum_{\beta\le\alpha}{\alpha\choose\beta} h_\ell^{|\beta-\alpha|}
\|\partial^{\beta}\big((\tilde u_p-\tilde u_\ell)\circ \varphi_\ell^{-1}\big)\|_{L^2(U_\ell^\varphi)}.$$
Let $p\in I_\ell\setminus\{\ell\}$. By \eqref{uellint},
$$
\tilde u_p|_{X_p\cap X_\ell}=\hu|_{J_p\cap J_\ell}= \tilde u_\ell|_{X_p\cap X_\ell}.$$
Hence, it follows from \eqref{sam2} and \Aref{A3} that for $|\beta|<s+\kappa$,
\begin{equation}\label{useA3}
\|\partial^{\beta}\big((\tilde u_p-\tilde u_\ell)\circ \varphi_\ell^{-1}\big)\|_{L^2(U_\ell^\varphi)}
\le C h_\ell^{s+\kappa-|\beta|}
|(\tilde u_p-\tilde u_\ell)\circ\varphi_\ell^{-1}|_{H^{s+\kappa}(U_\ell^\varphi)}.
\end{equation}
%
%
%
By Lemma~\ref{admis},
$$
|(\tilde u_p-\tilde u_\ell)\circ\varphi_\ell^{-1}|_{H^{s+\kappa}(U_\ell^\varphi)}
\le C\|\tilde u_p-\tilde u_\ell\|_{H^{s+\kappa}(\cM)}.$$
By \eqref{uelln}, 
$$
\|\tilde u_p-\tilde u_\ell\|_{H^{s+\kappa}(\cM)}\le 
C(\|L\tilde u_p\|_{K_L}+\|L\tilde u_\ell\|_{K_L}).$$
By combining these estimates, since $h_\cA\le 1$
and hence $h_\ell^{s+\kappa-|\alpha|}\le1$, we obtain 
\begin{equation}\label{utmutl}
\|(\tilde u-\tilde u_\ell)\circ \varphi_\ell^{-1}\|_{H^{s+\kappa}(U_\ell^\varphi)}
\le C\sum_{p\in I_\ell}\|L\tilde u_p\|_{K_L}.
\end{equation}
This implies
$$
\sum_{\ell=1}^m
\|L(\tilde u-\tilde u_\ell)\circ \varphi_\ell^{-1}\|_{H^{s-\kappa}(U_\ell^\varphi)}^2
\le C 
\sum_{\ell=1}^m\|L\tilde u_\ell\|_{K_L}^2.$$
Summarizing the estimates, we get 
$$
\|L\tilde u\|_{H^{s-\kappa}(\cM)}
\le C
\Big(\sum_{\ell=1}^m \|L\tilde u_\ell\|_{K_L}^2\Big)^{1/2},$$
and hence by \eqref{b3},
\begin{equation}\label{b4}
\|L\tilde u\|_{H^{s-\kappa}(\cM)}
\le C_1\|f\|_{H^{s-\kappa}(\cM)}+C_2\eps_1h_\cA^{-s+\kappa+d/2}.
\end{equation}

Finally, we will estimate Sobolev norms of $L u-L \tilde u$ by applying Lemma~\ref{sampling} locally, which in turn will give us estimates for $u-\tilde u$ thanks to
\eqref{reg}.
Since $X=\cup_{\ell=1}^mX_\ell$, it follows from \eqref{eps4} that
\begin{equation}\label{eps12}
\|L \tilde u|_{X}-L u|_{X}\|_2\le\Big(\sum_{\ell=1}^m\|L \tilde u|_{X_\ell}-L u|_{X_\ell}\|_2^2\Big)^{1/2}
\le \eps_1 + \eps_2,
\end{equation}
where 
$$
\eps_2:=\Big(\sum_{\ell=1}^m\|L \tilde u|_{X_\ell}-L \tilde u_\ell|_{X_\ell}\|_2^2\Big)^{1/2}.$$
By \eqref{giell} and \eqref{uint},
$$
L \tilde u_\ell|_{X_\ell}= \Wl \hat u|_{J_\ell}=\Wl\tilde u|_{X_\ell}.$$
By construction, $\tilde u_\ell$ is the kernel interpolant to $\tilde u$ on $X_\ell$, compare \eqref{uellint} and
\eqref{uint}. Hence by \eqref{lde1},
$$
\|L \tilde u|_{X_\ell}- L \tilde u_\ell|_{X_\ell}\|_\infty
 =\|L \tilde u|_{X_\ell}- \Wl\tilde u|_{X_\ell}\|_\infty\le
Ch_\ell^{s-\kappa-d/2}
|(\tilde u-\tilde u_\ell)\circ\varphi_\ell^{-1}|_{H^{s+\kappa}(U_\ell^\varphi)}.$$
In view of \eqref{utmutl},
\begin{align*}
\eps_2^2
\le \nu \sum_{\ell=1}^m\|L \tilde u|_{X_\ell}-L \tilde u_\ell|_{X_\ell}\|_\infty^2
\le C h_\cA^{2s-2\kappa-d}
\sum_{\ell=1}^m\|L\tilde u_\ell\|_{K_L}^2.
\end{align*}
and by \eqref{b3} we get 
\begin{equation}\label{eps2}
\eps_2\le \cE_{s-\kappa}(\cA)\big(
C_1h_\cA^{s-\kappa-d/2}\|Lu\|_{H^{s-\kappa}(\cM)}+C_2\eps_1\big).
\end{equation}
Let $v=L \tilde u-Lu$. Then $v\in H^{s-\kappa}(\cM)$ and
$$
\eps_3:=\|v\|_{H^{s-\kappa}(\cM)}\le \|L\tilde u\|_{H^{s-\kappa}(\cM)}+\|Lu\|_{H^{s-\kappa}(\cM)}.$$
Hence by \eqref{b4}, 
\begin{equation}\label{eps3}
\eps_3\le \cE_{s-\kappa}(\cA)\big(C_1\|Lu\|_{H^{s-\kappa}(\cM)}+C_2\eps_1h_\cA^{-s+\kappa+d/2}\big).
\end{equation}
By Lemma~\ref{sampling}, for any integer $r$ with $0\le r< s-\kappa$, and each $\ell=1,\ldots,m$,
$$
|v\circ \varphi_\ell^{-1}|_{H^r(U_\ell^\varphi)}
\le Ch_\ell^{-r}
\big(h_\ell^{d/2}\|v|_{X_\ell}\|_\infty
+h_\ell^{s-\kappa}|v\circ \varphi_\ell^{-1}|_{H^{s-\kappa}(U_\ell^\varphi)}\big).$$
Hence by Lemma~\ref{admis},
 \begin{align*}
\|v\|_{H^r(\cM)}^2&\le C\sum_{\ell=1}^m\|v\circ \varphi_\ell^{-1}\|_{H^r(U_\ell^\varphi)}^2\\
&\le C
\Big(\sum_{\ell=1}^m h_\ell^{d-2r}\|v|_{X_\ell}\|_\infty^2
+h_\cA^{2(s-\kappa-r)}\sum_{\ell=1}^m |v\circ \varphi_\ell^{-1}|_{H^{s-\kappa}(U_\ell^\varphi)}^2\Big)\\
&\le C_1h_\cA^{d-2r} \|v|_{X}\|_2^2
+C_2h_\cA^{2(s-\kappa-r)}\|v\|_{H^{s-\kappa}(\cM)}^2.
\end{align*}
By \eqref{eps12} and definitions of $v$ and $\eps_3$, we get
$$
\|Lu-L\tilde u\|_{H^{r}(\cM)}\le C_1h_\cA^{d/2-r} (\eps_1+\eps_2)
+C_2h_\cA^{s-\kappa-r}\eps_3.$$

It follows in view of \eqref{reg} that for any $0\le r< s-\kappa$,
\begin{equation}\label{eps123}
\|u-\tilde u\|_{H^{r+2\kappa}(\cM)}\le
C_1h_\cA^{d/2-r} (\eps_1+\eps_2)
+C_2h_\cA^{s-\kappa-r}\eps_3.
\end{equation}
By  \eqref{eps}, \eqref{eps2} and \eqref{eps3}, 
\begin{align*}
\eps_1+\eps_2&\le 
C_1\cE_{s}(\cA)h_\cA^{s-2\kappa-d/2}\|u\|_{H^s(\cM)}
+C_2\cE_{s-\kappa}(\cA)h_\cA^{s-\kappa-d/2}\|Lu\|_{H^{s-\kappa}(\cM)},\\
\eps_3&\le \cE_{s-\kappa}(\cA)\big(C_1\|Lu\|_{H^{s-\kappa}(\cM)}+C_2\cE_{s}(\cA)h_\cA^{-\kappa}\|u\|_{H^s(\cM)}\big).
\end{align*}
Hence by \eqref{eps123}, 
$$
\|u-\tilde u\|_{H^{r+2\kappa}(\cM)}\le
\cE_{s-\kappa}(\cA)\big(C_1\cE_s(\cA)
h_\cA^{s-2\kappa-r}\|u\|_{H^s(\cM)}
+C_2h_\cA^{s-\kappa-r}\|Lu\|_{H^{s-\kappa}(\cM)}\big).$$
By using \eqref{reg} again, we arrive at
\begin{equation}\label{errorE}
\|u-\tilde u\|_{H^{r+2\kappa}(\cM)}\le
\cE_{s-\kappa}(\cA)
\big(C_1\cE_s(\cA)
h_\cA^{s-2\kappa-r}\|u\|_{H^s(\cM)}
+C_2h_\cA^{s-\kappa-r}\|u\|_{H^{s+\kappa}(\cM)}\big),
\end{equation}
which implies  \eqref{Hb} in view of Lemma~\ref{extce}. 

To show  \eqref{infb} we use \eqref{uint} and Sobolev embedding of $H^{2\kappa+r}(\cM)$ into $C(\cM)$ as soon as $2\kappa+r>d/2$. 
It is easy to see that $r_0$ in \eqref{infb} is the smallest nonnegative integer $r$ satisfying this inequality.
\end{proof}

Note that convergence in the norm of $H^{2\kappa+r}(\cM)$ as $h_\cA\to0$ follows from \eqref{Hb} only when $0\le r<s-2\kappa-d$, thus under 
the  assumption
$$
s>2\kappa+d,$$ 
which is stricter than \eqref{As},
and the discrete convergence follows from \eqref{infb} under the assumption that 
$$
s>\max\{2\kappa,\tfrac{d}{2}+1\}+d.$$
In particular, in the discrete case in order for the power of $h_\cA$ to be positive, we need
$$
s>2\kappa+d+r_0=\max\{2\kappa+d,d+\lfloor \tfrac{d}{2}\rfloor+1\},$$
which is equivalent to $s>\max\{2\kappa,\tfrac{d}{2}+1\}+d$ since $s$ is integer.

For example, for $\kappa=1$ and $d\le 3$ we obtain the Sobolev norm convergence with the order
\begin{equation}\label{Hb3}
\|u-\tilde u\|_{H^{2+r}(\cM)}=\cO(h_\cA^{s-d-2-r})\quad\text{if}\quad s\ge d+3,\quad 0\le r\le s-d-3,
\end{equation}
and the discrete norm convergence 
\begin{equation}\label{infb3}
\|u|_X-\hu\|_{\infty}=\cO(h_\cA^{s-d-2})\quad\text{if}\quad s\ge d+3.
\end{equation}

\begin{remark}\label{eval}
\rm
The function $\tilde u$ of the proof of Theorem~\ref{error_bound} may be used for the approximate 
evaluation of the solution and derived quantities such as gradient at any points in $\cM$,
which however requires construction of a suitable partition of unity $\Gamma$. Instead, for the same purpose
we may employ the local interpolation functions $\tilde u_\ell$ determined by \eqref{uellint}. Indeed, since 
$\tilde u_\ell$ is a kernel interpolant to $\tilde u$ on $X_\ell$, the sampling inequality \eqref{saminf} applied to
$v=(\tilde u-\tilde u_\ell)\circ \varphi_\ell^{-1}$ on $\Omega=U_\ell^\varphi$ with $r=s$ leads to the estimate
$$
|\tilde u(x)-\tilde u_\ell(x)|\le C h_\cA^{s-d/2}\|\tilde u\|_{H^s(\cM)},\quad x\in U_\ell,$$
where the factor $C$ is independent of $u$ and $h_\cA$, compare the proof of \eqref{lde2} in the above. By \eqref{b4} and
\eqref{eps} we have
$$
\|\tilde u\|_{H^s(\cM)}\le C h_\cA^{-\kappa}\|u\|_{H^s(\cM)}.$$
By taking into account \eqref{infb}, we arrive at the bound
\begin{equation}\label{infbell}
|u(x)-\tilde u_\ell(x)|\le C_1h_\cA^{s-2\kappa-d-r_0}\|u\|_{H^s(\cM)}
+C_2h_\cA^{s-\kappa-d/2-r_0}\|u\|_{H^{s+\kappa}(\cM)}, \quad x\in U_\ell,
\end{equation}
which shows that $\tilde u_\ell$ may be used as approximation of $u$ on $U_\ell$ as soon as 
$s>\max\{2\kappa,\tfrac{d}{2}+1\}+d$.
\end{remark}

%
%
%


\begin{remark}\label{chains}
\rm Condition \Aref{A3} prohibits small non-empty intersections $X_p\cap X_\ell$ because the images of any such set under 
$\varphi_p$ and $\varphi_\ell$ are required to be $(s+\kappa)$-determining sets. In fact, we only need that sufficiently
many intersections have this property. More precisely, we may replace \Aref{A3} by the following assumption:

\begin{enumerate}
\item[(\chain)]
Given $\sigma\in\NN$, $\chi<\infty$, $R<\infty$ and $\lambda<\infty$, assume that for each pair $\ell,p$ 
with $U_\ell\cap U_p\ne\emptyset$, there is a set $U_{\ell,p}\subset\cM$ and a chain $U_{p_0},U_{p_1},\ldots,U_{p_S}$,
with $p_0=\ell$, $p_S=p$ and $S\le\sigma$, such that
\begin{itemize} 
\item
$\bigcup_{i=0}^S U_{p_i}\subset U_{\ell,p}\subset \hat U_k^c$ for some $1\le k\le\hat n$, 
\item
$\chi(\hat\varphi_k(U_{\ell,p}))\le \chi$, 
\item
$\diam(\hat\varphi_k(U_{\ell,p}))\le R\max\big\{\!\diam(U_\ell^\varphi),\diam(U_p^\varphi)\big\}$,
\item
$\lambda_{s+\kappa}\big(\hat\varphi_k(U_{\ell,p}),\hat\varphi_k(X_{p_i}\cap X_{p_{i-1}})\big)\le\lambda$,
$i=1,\ldots,S$,\quad and
\item in the case $\ell=p$ we choose $S=1$, $U_{\ell,p}=U_\ell$ and $k$ such that $\varphi_\ell=\hat\varphi_k|_{U_\ell}$.
\end{itemize}
\end{enumerate}
The last condition ensures that  \Aref{A2} follows from  (\chain), with  $\chi_\cA\le\chi$, and that 
$\lambda_{s+\kappa}(U^\varphi_\ell,X^\varphi_\ell)\le\lambda$.
If we assume (\chain) instead of  \Aref{A2} and \Aref{A3}, then Theorem~\ref{error_bound} still holds, with $C_1$ and $C_2$
depending in addition on $\sigma$ and $R$. Indeed, we apply \Aref{A3} in the proof twice, 
namely to show \eqref{ldep} and \eqref{useA3}. 
While \eqref{ldep} uses the case  $\ell=p$, where there is no difference between (\chain) and \Aref{A3},
in order to replace \eqref{useA3} we write
$$
\tilde u_p-\tilde u_\ell=\sum_{i=1}^S \tilde u_{p_i}-\tilde u_{p_{i-1}},$$
and deduce
\begin{align*}
\|\partial^{\beta}\big((\tilde u_p-\tilde u_\ell)\circ \varphi_\ell^{-1}\big)\|_{L^2(U_\ell^\varphi)}
&\le C\|\partial^{\beta}\big((\tilde u_p-\tilde u_\ell)\circ \hat\varphi_k^{-1}\big)\|_{L^2(\hat\varphi_k(U_{\ell,p}))}\\
&\le C\sum_{i=1}^S
\|\partial^{\beta}\big((\tilde u_{p_i}-\tilde u_{p_{i-1}})\circ \hat\varphi_k^{-1}\big)\|_{L^2(\hat\varphi_k(U_{\ell,p}))}.
\end{align*}
In view of \eqref{uellint}, 
$$
\tilde u_{p_i}|_{X_{p_i}\cap X_{p_{i-1}}}=\hu|_{J_{p_i}\cap J_{p_{i-1}}}= 
\tilde u_{p_{i-1}}|_{X_{p_i}\cap X_{p_{i-1}}},\quad i=1,\ldots,S,$$
and hence by \eqref{sam2} and (\chain) we have for $|\beta|<s+\kappa$ and all $i=1,\ldots,S$,
$$
\|\partial^{\beta}\big((\tilde u_{p_i}-\tilde u_{p_{i-1}})\circ \hat\varphi_k^{-1}\big)\|_{L^2(\hat\varphi_k(U_{\ell,p}))}
\le C h_{\ell,p}^{s+\kappa-|\beta|}
|(\tilde u_{p_i}-\tilde u_{p_{i-1}})\circ \hat\varphi_k^{-1}|_{H^{s+\kappa}(\hat\varphi_k(U_{\ell,p}))},$$
where $h_{\ell,p}:=\diam(\hat\varphi_k(U_{\ell,p}))$, which may serve as a replacement
for \eqref{useA3} in the arguments leading to \eqref{utmutl} since $h_{\ell,p}\le R\,h_\cA$ by (\chain), and as in
 the proof of Lemma~\ref{admis} it can be shown that
$$
|(\tilde u_{p_i}-\tilde u_{p_{i-1}})\circ \hat\varphi_k^{-1}|_{H^{s+\kappa}(\hat\varphi_k(U_{\ell,p}))}\le
C\|\tilde u_{p_i}-\tilde u_{p_{i-1}}\|_{H^{s+\kappa}(\cM)},\quad i=1,\ldots,S.$$

\end{remark}

\begin{remark}\label{pnorm_est}
\rm Theorem~\ref{error_bound} remains valid if we define 
$\hat u\in\RR^n$ by minimizing a $p$-norm \eqref{pmin}
instead of least squares as in \eqref{leastsq}.
Indeed, the proof goes through with obvious changes. In particular, to show \eqref{eps}
with 
$$
\eps_1:=\|W\!M\hat u-Mf|_X\|_p\le \|W\!Mu|_X-Mf|_X\|_p
=\big\|\big\{\|\Wl u|_{X_\ell} - Lu|_{X_\ell}\|_p\big\}_{\ell=1}^m\big\|_p$$
we need a $p$-norm version of Lemma~\ref{admis}, which holds with $\sqrt{\mu_\cA}$ replaced by $\mu_\cA^{1/p}$.  
\end{remark}


\section{Numerical examples}\label{numerics}
The goal of this section is numerical verification of the convergence orders suggested by the 
error bounds of Theorem~\ref{error_bound}, in particular by the discrete maximum norm bound of \eqref{infb3}. 

We consider the $d$-dimensional sphere $\cM=\SSS^d$ embedded into $\RR^{d+1}$,
$$
\SSS^d=\{x\in\RR^{d+1}:\|x\|_2=1\},$$ 
and solve the equation
\begin{equation}\label{testp}
Lu:=-\Delta_\cM u+u = f\quad \text{in}\quad \cM,
\end{equation}
w.r.t.\ $u:\SSS^d\to\RR$, where $\Delta_\cM$ is the spherical Laplace-Beltrami operator, and $f:\SSS^d\to\RR$ is given by
$$
f(x)=\big(4d+5-s^2(x)\big)\sin s(x)
+d\,s(x)\cos s(x),\quad x\in\SSS^d, $$
where
$
s(x):=2\sum_{i=1}^{d+1}x_i$.
The exact solution of \eqref{testp} is
$$
u(x)=\sin s(x) =\sin\Big(2\sum_{i=1}^{d+1}x_i\Big),\quad x\in\SSS^d. $$

The operator $L$ in \eqref{testp} is the special case of \eqref{LB} with $\alpha=\kappa=1$. As discussed in Section~\ref{difop}, it satisfies
all assumptions required for the least squares method of Section~\ref{method} and for the error bounds of 
Theorem~\ref{error_bound}, as soon as $K$ is a zonal reproducing kernel for the Sobolev space $H^s(\cM)$ for sufficiently
large $s$.

We use the (scaled) Mat\'ern kernels 
$$
K(x,y)=M_{s,d}^\eps(x,y)=M_{s,d}(\eps x,\eps y),\qquad x,y\in \SSS^d,$$
where $M_{s,d}$ is given in \eqref{Matern}. As noticed in Section~\ref{pdk}, $M_{s,d}^\eps$ restricted to $\cM$ is a 
reproducing kernel for $H^s(\cM)$ whenever $s>d/2$. We discussed in Section~\ref{difop} that these kernels are
zonal and satisfy the hypotheses of Lemma~\ref{KL} with respect to the operator $L$. 
Since the order of $L$ is two, the method of Section~\ref{method} is well defined as soon as $s>d/2+1$. The assumption 
\eqref{As} used in the error bounds is satisfied whenever $s$ is an integer greater than $d/2+2$, in particular
$s\ge 3$ for $d=1$, $s\ge 4$ for $d=2$ or $3$.

In order to compute the differentiation matrices \eqref{Wl} we need evaluations of the kernel $K_L(x,y)=L_1K(x,y)$ for
arbitrary $x,y\in \SSS^d$. For this, we use the following general formula for the application of
the Laplace-Beltrami operator $\Delta_\cM$ to the first argument of any  zonal kernel $K$ in the \emph{$f$-form} 
 $K(x,y)=f(t)$, where $t=\|x-y\|_2^2/2=1-x^Ty$, $x,y\in \SSS^d$,
\begin{equation}\label{LBKfform}
K_{\Delta_\cM}(x,y)=t(2-t)f''(t) + d(1-t)f'(t).
\end{equation}
The formula can be easily shown by using \cite[Lemma 1.4.2]{DaiXu13}. For more information on the $f$-form, see \cite{SchabackMATLAB}
and \cite{mFDlab}.

For $d\in\{2,3\}$ we generate rather regular  sets of nodes on the unit sphere $\SSS^d$ by first producing a
quasi-random set with uniform probability distribution on $\SSS^d$, and then applying a thinning algorithm 
to increase the minimum separation of the nodes. More precisely, for a given $\delta>0$, we first generate $N_0$ 
points in $\RR^{d+1}$ with normally distributed components, where $N_0$ is obtained by rounding $\vol_d(\SSS^d)/\delta^d$ 
to the nearest integer. We then project these points to the sphere  $\SSS^d$, and run on them MATLAB's command \texttt{uniquetol} 
with absolute tolerance $2\delta$, to produce the final set $X$ whose cardinality will be denoted by $N$. 
 For each  $d\in\{2,3\}$ we create five sets of this type obtained with $\delta=4^{-i/d}\delta_0$, $i=0,1,\ldots,4$, where we
choose $\delta_0=0.1$ in 2D and $\delta_0=0.15$ in 3D. The cardinalities $N$ of the resulting sets $X$  can be found in
Tab.~\ref{NhA} below. An example of a node set $X$ obtained this way for the sphere $\SSS^2$ is shown in Fig.~\ref{setX}.

\begin{figure}[!ht]
\begin{center}
\includegraphics[width=7.5cm]{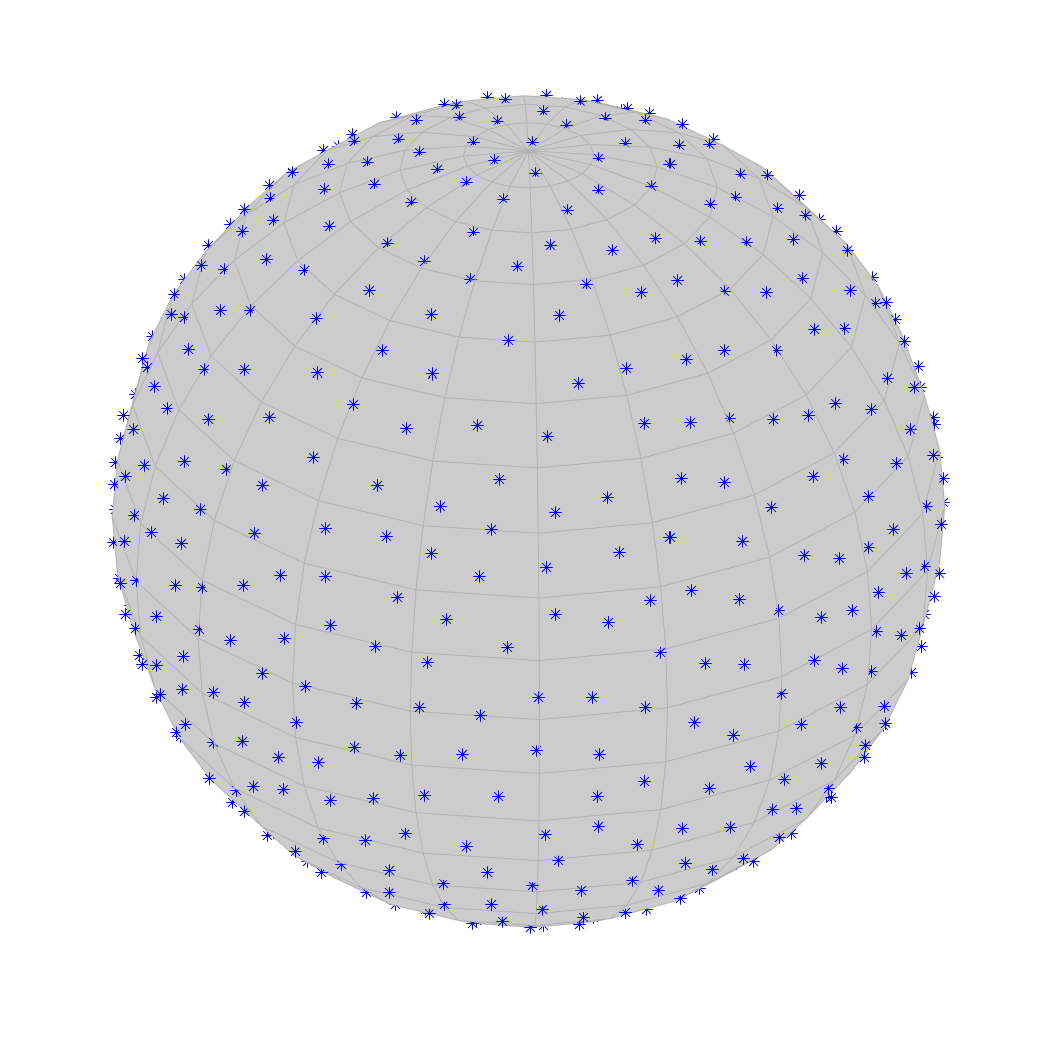}
\caption{Example of a node set $X$ in $\SSS^2$: 604 nodes obtained with $\delta=0.05$.
}
\label{setX}
\end{center}
\end{figure}

We run our method using the kernels $M_{s,d}^\eps$, $s=3,\ldots,7$, with the shape parameter $\eps$  reported in Tab.~\ref{snelle}. We choose the
set $X_\ell$ of \eqref{XA} for each $\ell=1,\ldots,N$ by collecting $n_\ell=2\dim\Pi^d_{s-1}$ nearest neighbors of the node
$x_\ell$ in $X$ with respect to the Euclidian distance in $\RR^{d+1}$. The numbers $n_\ell$ are also included in
Tab.~\ref{snelle}. In all cases these numbers were big enough in our experiments so that the approximation  order would  not
increase if $n_\ell$ were increased, even though Conditions \Aref{A3} or (\chain) are not guaranteed with this choice of
$X_\ell$.

\begin{table}[htbp!]\scriptsize 
\centering
\renewcommand{\arraystretch}{1.25}
\begin{tabular}{|c||c|c||c|c|c|c|c||c|c|}
\hline
& \multicolumn{2}{c||}{$\SSS^2$} & \multicolumn{2}{c|}{$\SSS^3$}\\
\hline
$s$ &  $n_\ell$ & $\eps$ & $n_\ell$ &  $\eps$\\
\hline\hline
3  & 6 & $2^{-6}$ & 8 & $2^{-9}$\\
\hline
4  & 12 & 0.5 & 20 & $2^{-4}$\\
\hline
5  & 20 & 2 & 40 & 0.25\\
\hline
6  & 30 & 3 & 70 & 1\\
\hline
\end{tabular}
\caption{The number of nearest neighbors $n_\ell$ used to generate the sets $X_\ell$ and the shape parameter $\eps$ of the kernel
$M_{s,d}^\eps$, chosen individually for each $s=3,\ldots,7$ in both 2D and 3D.
}
\label{snelle}
\end{table}

Recall that the error bound \eqref{infb3} is given in terms of the maximum diameter $h_\cA$ of the sets $U^\varphi_\ell$ that depend
on the choice of the atlas $\cA$. However, we do not generate either the sets $U_\ell$, the atlas $\cA$, or
the partition of unity $\Gamma$ as they are not needed for the computation of the discrete approximate solution $\hat u$.
Therefore we replace $h_\cA$ by the easily computable quantity
$$
h_\cX:=\max_{1\le \ell\le N}\tilde h_\ell,$$
where $\tilde h_\ell$ denotes the Euclidean diameter of $X_\ell\subset\RR^{d+1}$. By choosing the primary atlas
$\hat\cA$ for example as a number of half-spheres projected to appropriate tangent hyperplanes to $\SSS^d$ in $\RR^{d+1}$, we
may achieve that the quantities $h_\cA$ and  $h_\cX$ are equivalent in order when one of them goes to zero, for any 
admissible atlas $\cA$. 
The numbers $h_\cX$ depending on $d$,  $N$ and $s$ are reported in Tab.~\ref{NhA}. 
We have also estimated the quasi-uniformity of $\cX$ and $X$ by computing 
$$
q_\cX:=\max_{1\le i,j\le N}\frac{\tilde h_i}{\tilde h_j},\qquad 
q_{X,\cA}:=\max_{1\le \ell\le N}\frac{h_\cX}{2n_\ell^{1/d}\delta_\ell}$$
as replacement for $q_\cA$ and $q_{X,\cA}$. In all experiments both quantities were below 2.1.

\begin{table}[htbp!]\scriptsize 
\centering
\renewcommand{\arraystretch}{1.25}
\begin{tabular}{|c||c|c|c|c|c||c|c|c|c|}
\hline
\multirow{2}{6pt}{$N$} & \multicolumn{4}{|c|}{$s$}\\ 
\cline{2-5} & 3 & 4 & 5 & 6 \\
\hline
152   & 0.82 & 1.19 & 1.40 & 1.62 \\
604   & 0.45 & 0.61 & 0.78 & 0.90  \\
2,414 & 0.22 & 0.31 & 0.39 & 0.48  \\
9,560 & 0.12 & 0.16 & 0.21 & 0.25 \\
38,358& 0.06 & 0.08 & 0.10 & 0.13  \\
\hline
\end{tabular}
\hspace{48pt}
\begin{tabular}{|c||c|c|c|c|c||c|c|c|c|}
\hline
\multirow{2}{6pt}{$N$} & \multicolumn{4}{c|}{$s$}\\
\cline{2-5} & 3 & 4 & 5 & 6 \\
\hline
257     & 1.06 & 1.41 & 1.68 & 1.90\\
1,014   & 0.73 & 0.93 & 1.16 & 1.35\\
4,048   & 0.45 & 0.60 & 0.74 & 0.88\\
16,276  & 0.29 & 0.39 & 0.49 & 0.57\\
65,164  & 0.19 & 0.25 & 0.31 & 0.36\\
\hline
\end{tabular}
\medskip

(a) $\SSS^2$ \hspace{160pt} (b) $\SSS^3$
\caption{$h_\cX$: maximum diameter of subsets $X_\ell$ for (a) $\SSS^2$ and (b) $\SSS^3$, depending on $N$ and $s$. 
}
\label{NhA}
\end{table}

We plot the error  $\|u|_X-\hat u\|_\infty$ against $h_\cX$ in Fig.~\ref{error2D}
for $\SSS^2$ and in Fig.~\ref{error3D} for $\SSS^3$. The plots indicate convergence orders between $h_\cX^{s-3}$ and
$h_\cX^{s-2}$ for both $d=2,3$, which is higher than the theoretical estimate \eqref{infb3} that justifies $h_\cX^{s-4}$
in 2D and $h_\cX^{s-5}$ in 3D. Numerical convergence is observed for $s\ge3$ in 2D and  $s\ge4$ in 3D.

\begin{figure}[!ht]
\begin{center}
\includegraphics[width=10cm]{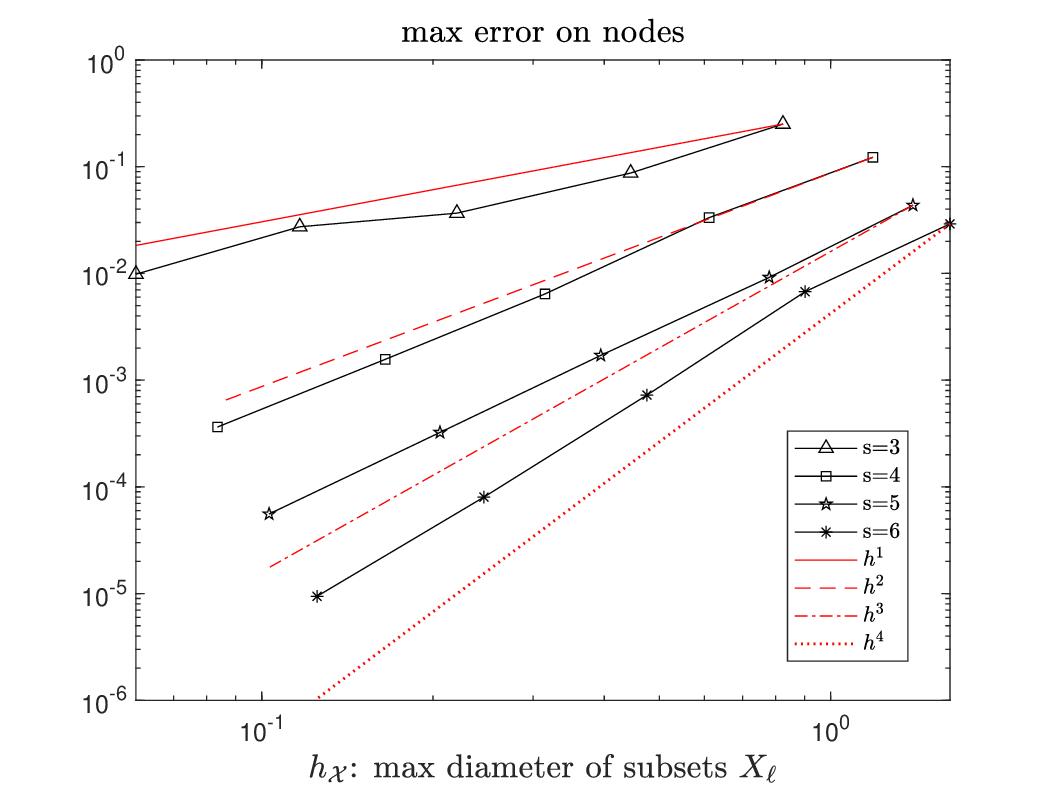}
\caption{Maximum error $\|u|_X-\hat u\|_\infty$ in 2D as function of $h_\cX$. For comparison we also include straight lines with slopes
corresponding to convergence orders between $h_\cX$ and $h^4_\cX$.
}
\label{error2D}
\end{center}
\end{figure}

\begin{figure}[!ht]
\begin{center}
\includegraphics[width=10cm]{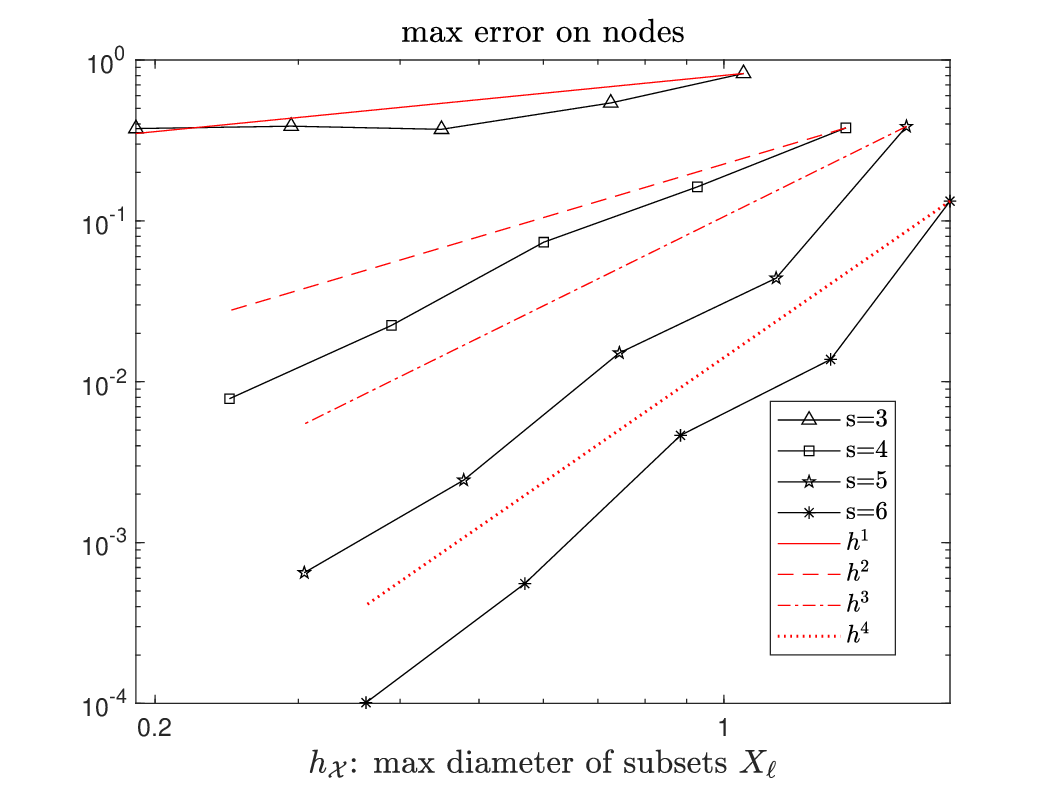}
\caption{Maximum error $\|u|_X-\hat u\|_\infty$ in 3D as function of $h_\cX$.
}
\label{error3D}
\end{center}
\end{figure}

Note that the values for $\eps$ in Tab.~\ref{snelle} are chosen
as small as possible such that the results are not yet affected by the rounding errors.
In most cases we choose $\eps$ as  an integer power of two, $\eps=2^j$ for some $j\in\ZZ$. 
However, in the case $s=6$ in 2D we make an exception and take
$\eps=3$ because $\eps=2$ gives errors severely affected by rounding for the
largest set $X$ produced with $\delta=1/16$, whereas $\eps=4$ leads to solutions whose errors are too high, 
see Fig.~\ref{error2Ds6}. 
This figure illustrates that the errors tend to be smaller for smaller $\eps>0$, without changing
the approximation order. However, when we decrease $\eps$, we get to the point that  the kernel matrices $M_{s,d}^\eps|_{X_\ell}$
become too ill-conditioned, affecting the computation of $W_\ell$ in \eqref{Wl}, 
which leads to a high error of the solution on denser sets $X$, as we see for $\varepsilon\in\{1,2\}$ in the figure.

\begin{figure}[!ht]
\begin{center}
\includegraphics[width=10cm]{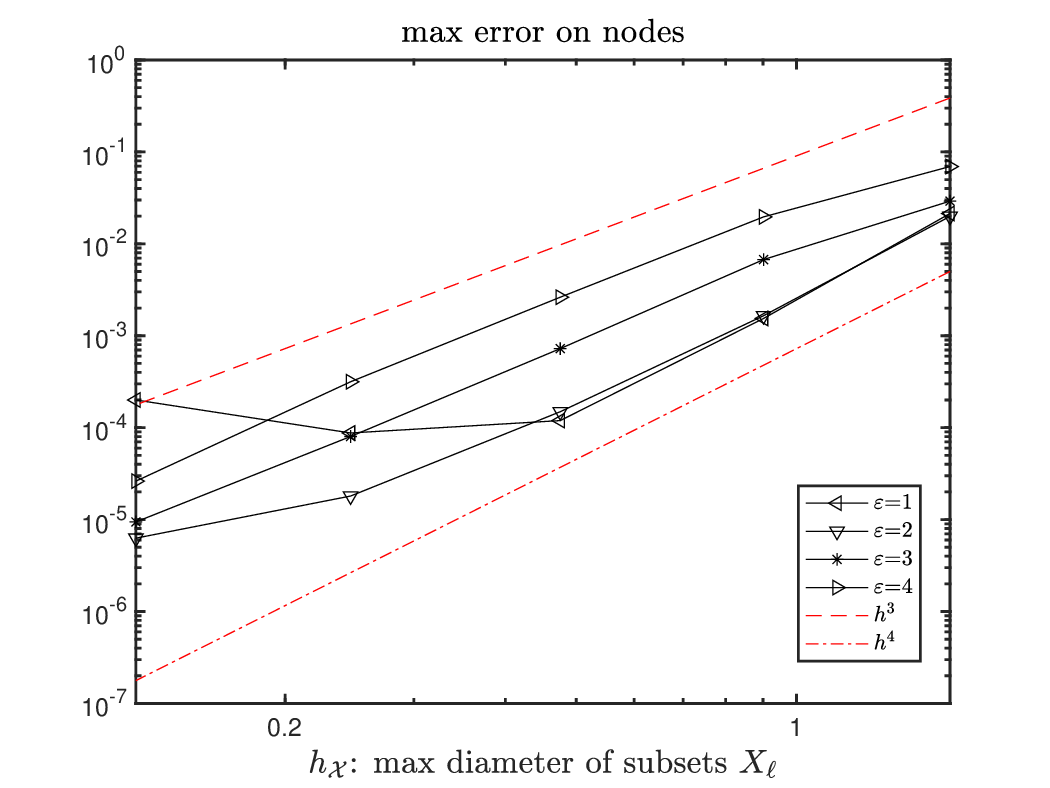}
\caption{Maximum error $\|u|_X-\hat u\|_\infty$ in 2D for $s=6$ and several choices of $\eps$.
}
\label{error2Ds6}
\end{center}
\end{figure}

\section{Conclusion}\label{conclusion}

In this paper we introduced a discrete least squares version of the meshless finite difference method for elliptic
differential equations on closed manifolds,
based on invertible local differentiation matrices for overlapping subsets of nodes. We obtained
error bounds for this method by a new technique, and tested its convergence properties numerically.
The sparse overdetermined linear system this method relies on has full rank without any additional assumptions like sufficient density of nodes.
In particular, it is not related to a discretization of any variational or continuous least squares formulation of the 
differential equation. Since no exact or approximate integration is involved, the  scheme does not require any partition of
the manifold or its ambient space and is purely meshless.

Numerical examples suggest better convergence orders of the method than predicted by the error bounds, which indicates that
theoretical results may be improved by refining the techniques introduced in the proof of Theorem~\ref{error_bound}. 
In particular, we expect that the factor $h^{-d/2}$ in the estimate of Lemma~\ref{extce} may be removed, which would allow to
somewhat improve the bounds \eqref{Hb} and \eqref{infb}, but will not be sufficient for explaining numerical results.
The assumption that the nodes are quasi-uniform in the sense that $q_{X,\cA}$ is uniformly bounded does not seem natural, but
is essential for the proof.

New ideas are needed in order to extend the method with guaranteed full rank of the system matrix to boundary value problems, 
conditionally positive definite kernels, or to the case when the elliptic operator $L$ is not necessarily self-adjoint with
respect to the inner product of $H^s(\cM)$ generated by the kernel $K$. 
In particular, conditionally positive definite kernels may help to improve the conditioning of the least squares problem
\eqref{ods2}, as they are known to do in the RBF-FD method on domains in $\RR^d$ \cite{BFFB17}, and remove the need for tuning 
the shape parameter $\varepsilon$ we observed in Section~\ref{numerics}.


\bigskip

\noindent
\textbf{Conflict of Interest Statement.} The author has no competing interests to
declare that are relevant to the content of this article.

\bibliographystyle{abbrv} %
\bibliography{meshless}

\end{document}